 \documentclass[reqno,11pt]{amsart}
 \usepackage{amsthm, a4, latexsym,comment}
 
\usepackage[unicode]{hyperref}
\hypersetup{
    colorlinks=true, 
    linktoc=all,     
    linkcolor=blue,  
}
\usepackage{srcltx}
\usepackage{mathabx}
\usepackage{amsmath}
\usepackage{amsfonts}
\usepackage{amssymb}
\usepackage{verbatim}

\usepackage[T1]{fontenc}
\usepackage[utf8]{inputenc}

\usepackage{xcolor} 

\usepackage{amstext,bbm,nicefrac,physics, tikz}
\usepackage{subcaption}

\usetikzlibrary{calc}
\usetikzlibrary{decorations.pathmorphing}

\def\@rmrk#1#2{\refstepcounter
  {#1}\@ifnextchar[{\@yrmrk{#1}{#2}}{\@xrmrk{#1}{#2}}}

\makeatletter\@addtoreset{equation}{section}\makeatother

 \newfont{\bfit}{cmbxti10 scaled 1200}

\renewcommand{\d}{{\rm{d}}}
 \newcommand{\e}{{\rm e} }

 \newcommand{\eps}{\varepsilon}

 \newcommand{\R}{\mathbb{R}}

 \newcommand{\N}{\mathbb{N}}
 \newcommand{\Z}{\mathbb{Z}}
 
 \newcommand{\Sym}{\mathfrak{S}}

 \newcommand{\E}{\mathbb{E}}
 
 \renewcommand{\P}{\mathbb{P}}
 \def\1{{\mathchoice {1\mskip-4mu\mathrm l} 
{1\mskip-4mu\mathrm l}
{1\mskip-4.5mu\mathrm l} {1\mskip-5mu\mathrm l}}}

\newcommand{\Ccal}{{\mathcal C}}
 \newcommand{\Hcal}{{\mathcal H}}

 \newcommand{\Xcal}{{\mathcal X}}
 
 \newcommand{\be}{\alpha}
 
  \newcommand{\smfrac}[2]{\textstyle{\frac {#1}{#2}}}

\newcommand\FF{{\mathcal F}}

\newcommand\LL{{\mathcal L}}
\newcommand\MM{{\mathcal M}}

\newcommand\PP{{\mathcal P}}

\newcommand\PMF{{\PP\kern-2pt\MM\FF}}

\newcommand\PML{{\PP\kern-2pt\MM\LL}}

\renewcommand{\L}{\Lambda}

\newcommand{\ssup}[1] {{\scriptscriptstyle{({#1}})}}

{\nopagebreak {\hfill\rule{2mm}{2mm}}\\ }

\newtheorem{theorem}{Theorem}[section]
\newtheorem{lemma}[theorem]{Lemma}

\newtheorem{prop}[theorem]{Proposition}

\newtheoremstyle{theorem}{1.5ex}{1.5ex}{\itshape\rmfamily}{}
{\bfseries\rmfamily}{}{2ex}{}

\newcounter{remark}
\newenvironment{remark}{
  \refstepcounter{remark}
  \noindent {\bf Remark \theremark\ }\nopagebreak
}{
  \nopagebreak {\hfill\rule{2mm}{2mm}}\\
}

\def\thebibliography#1{\section*{References}
 \list%
 {\arabic{enumi}.}
  {\settowidth\labelwidth{[#1]}\leftmargin\labelwidth
  \advance\leftmargin\labelsep
  \parsep0pt\itemsep0pt
  \usecounter{enumi}}
  \def\newblock{\hskip .11em plus .33em minus .07em}
  \sloppy          
  \sfcode`\.=1000\relax}

\begin{document}

\title[Self-Repellent Brownian Bridges in an Interacting Bose Gas]
{\large Self-Repellent Brownian Bridges\\ \medskip in an Interacting Bose Gas}
\author{Erwin Bolthausen}
\address{Erwin Bolthausen, Institut f\"ur Mathematik, Universit\"at Z\"urich, Winterthurstrasse 190, Z\"urich Ch 8057, Switzerland}

\email{eb@math.uzh.ch}

\author{Wolfgang K\"onig}
\address{Wolfgang K\"onig, WIAS Berlin and TU Berlin, Mohrenstrasse 39, Berlin 10117, Germany}

\email{koenig@wias-berlin.de}

\author{Chiranjib Mukherjee}
\address{Chiranjib Mukherjee, Fachbereich Mathematik und Informatik, Universit\"at M\"unster, Einsteinstra{\ss}e 62, M\"unster 48149, Germany}

\email{chiranjib.mukherjee@uni-muenster.de}

\thispagestyle{empty}
\vspace{-0.5cm}



\date{\today}

\begin{abstract}
We consider a model of $d$-dimensional interacting quantum Bose gas, expressed in terms of an ensemble of interacting Brownian bridges in a large box and undergoing the influence of all the interactions between the legs of each of the Brownian bridges. We study the thermodynamic limit of the system and give an explicit formula for the limiting free energy and a necessary and sufficient criterion for the occurrence of a condensation phase transition. For $d\geq 5$ and sufficiently small interaction, we prove that the condensate phase is not empty. The ideas of proof rely on the similarity of the interaction to that of the self-repellent random walk, and build on a lace expansion method conducive to treating {\it paths} undergoing mutual repellence within each bridge.

\end{abstract}

\maketitle

\tableofcontents

\section{Introduction and main results}\label{sec-intro}

\subsection{Background.}\label{sec-Background}

\noindent The well-known {\em interacting Bose gas} can be written, using the {\em Feynman--Kac formula}, as a random ensemble of many Brownian bridges of random and unbounded lengths ($=$ particle numbers) with mutually repellent interactions between any two legs of any of the bridges. In the thermodynamic limit, the transition from absence to emergence of  macroscopically many particles in long bridges in this ensemble as the density increases is a prominent open problem that is strongly linked with the famous phase transition called {\em Bose--Einstein condensation (BEC)}, if not equal to that. 

The starting observation of the present paper is the following: Each of these bridges, under the influence of all the interactions between its legs alone, is a transformation of a Brownian bridge that should lie in the class of the well-known {\em self-repellent walk} (or {\em weakly self-avoiding walk}). This class of random motions is prominent in statistical physics and has been studied a lot since several decades, although a lot of problems are left widely open. 

However, to the best of our knowledge, a substantial connection between the Bose gas and the self-repellent walk has not been made. In this paper, we establish this connection 
to study the influence of the interactions between its legs and characterize the free energy of the system explicitly in terms of a variational formula and find criteria for the presence of a condensate phase transition in this model. These criteria lie at the heart of the critical properties of the self-repellent Brownian bridges. 

Below in Section~\ref{sec-model} we will introduce the mathematical layout of the model, in Section \ref{sec-Purpose} we will explain our purposes and formulate our main results in Section~\ref{sec-Results}. In Section~\ref{sec-background} we will discuss the necessary background, outline the method of proof and draw conclusions from our main results. The proofs will be given in the subsequent sections.

\subsection{The Model.}\label{sec-model}

\noindent Let us introduce the model that we are interested in. As will be explained in Section \ref{sec-BoseModel}, 
it is indeed a version of the interacting Bose gas. For the sake of simplicity, we prefer to 
formulate it as a model of a {\it random partition} of a positive integer $N\in \N$.

Indeed, denote the set of partitions of $N\in\N$ by
\begin{equation}\label{PNdef}
\mathfrak{P}_N=\Big\{l=(l_k)_{k \in \N}\in\N_0^\N\colon\sum_{k\in \N} kl_k=N\Big\}.
\end{equation}
With a parameter $\beta\in(0,\infty)$, introduce the following probability measure on $\mathfrak P_N$:
\begin{equation}\label{PNNewdef}
{\tt P}^{\ssup N}_{\beta,\L}(l)=\frac 1{ Z_N^{\ssup{{\rm bc}}}(\beta,\L)}\prod_{k\in\N}\frac{\big[|\L|\Gamma^{\ssup{\rm bc}}_{\L,k}\big]^{l_k}}{l_k!\,k^{l_k}},\qquad l=(l_k)_{k\in\N}\in\mathfrak P_N,
\end{equation}
where 
\begin{align}\label{cycleNew}
Z_N^{\ssup{{\rm bc}}}(\beta,\L)=
 \sum_{l\in\mathfrak{P}_N}  \prod_{k\in\N}\frac{\big[|\L|\Gamma^{\ssup{\rm bc}}_{\L,k}\big]^{l_k}}{l_k!\,k^{l_k}}
\end{align}
is the partition function of the model (the normalization). It is well-known that $N /\prod_{k \in\N} l_k! k^{l_k}$ is equal to the number of permutations of $1,\dots,N$ that have precisely $l_k$ cycles for any $k\in\N$. Thus, \eqref{PNdef}- \eqref{cycleNew} indeed define a model of a random permutation with multiplicative weight $|\L| \Gamma_{\L,k}^{\ssup{\rm bc}}$ for any cycle of length $k$; see Section~\ref{sec-RandPerm} for further details.

Let us now explain the terms appearing in \eqref{PNdef} -- \eqref{cycleNew}. First, $\Lambda\subset \R^d$ denotes a centered box, and $\Gamma^{\ssup{\rm bc}}_{\L,k}$ is defined as an integral of an exponential interaction with respect to a Brownian bridge of length $k$ in $\L$ (cf.\ \eqref{Gammabcdef} below). If we denote by $\mu_{x,y}^{\ssup{\rm{bc},\beta}}(\cdot)$ the canonical  Brownian bridge measure on the time interval $[0,\beta]$ from $x$ to $y$ subject to the boundary condition \lq bc\rq\ in the box $\Lambda$ 
(the boundary condition to be specified later), we set  
\begin{equation}\label{muLambdadef}
\mu_{\L}^{\ssup{{\rm bc},\beta}}(\d f)=\frac 1{|\L|} \int_{\L} \d x\, \mu^{\ssup{{\rm bc},\beta}}_{x,x}(\d f).
\end{equation}
Thus, $\mu_{\L}^{\ssup{{\rm bc},\beta}}(\cdot)$ is
a Brownian bridge measure on $\Ccal=\Ccal_1$, where $\Ccal_k$ denotes the set of continuous functions $[0, k\beta]\to\R^d$, with uniform starting-termination site in the centered box $\L$. If we write $\mu(f)$ for the integral $\int f \d\mu$ of a function $f$ with respect to a measure $\mu$, the weight $\Gamma^{\ssup{\rm bc}}_{\L,k}$ is defined as 
\begin{equation}\label{Gammabcdef}
\Gamma^{\ssup{\rm bc}}_{\L,k}=\mu_{\L}^{\ssup{{\rm bc},k\beta}}\Big[{\rm e}^{-\sum_{1\leq i<j\leq k}V(B_i,B_j)}\Big],
\end{equation}
where, for any two continuous functions $f, g \in \mathcal C$, we write 
\begin{equation}\label{Vdef}
V(f,g)=\int_0^\beta v(|f(s)-g(s)|)\,\d s,
\end{equation}
and $v\colon[0,\infty) \to [0,\infty)$ is a bounded measurable {\it pair interaction potential} having compact support. In \eqref{Gammabcdef}, the Brownian bridge $B$ in $\Ccal_k$ has been decomposed into its {\it $k$ legs}, defined as 
\begin{equation}\label{jthleg}
 B_j=(B_j(s))_{s\in[0,\beta]}=(B((j-1)\beta+s))_{s\in[0,\beta]} \in\Ccal, \qquad j\in[k]=\{1,\dots,k\}.
\end{equation}
Hence, the exponential interaction term in \eqref{Gammadef} is the pair-interaction sum of all the legs of a Brownian bridge of length $k$. 
Therefore, $\Gamma^{\ssup{\rm bc}}_{\L,k}$ is the partition function of what can be called a variant of the {\em self-repellent} (or {\em weakly self-avoiding}) Brownian bridge, where the usual $\delta$-interaction between any two $B(s)$ and $B(t)$ is replaced by a smooth interaction between any two legs $B_i$ and $B_j$. The starting site of this motion is not the origin as usual, but is uniformly distributed over a large box, and the endpoint is conditioned to be the initial site. We will be studying this model in the limit $N\to\infty$ with the box $\L=\L_N$ having volume $\sim N/\rho$ for some $\rho \in(0,\infty)$, i.e., in the thermodynamic limit.


By the Feynman--Kac formula, the above model can also be represented as an ensemble of $N$ Brownian bridges in $\L$ with time-horizon $[0,\beta]$ under some symmetrization condition, where each leg has an interaction with each other leg. Due to the symmetrization, one can decompose the ensemble into closed loops (bridges); see Section~\ref{sec-BoseModel} for further details. 

We remark that we have kept all features that come from the Bose gas in a large box, in particular the fact that the particles are confined to a box with certain boundary conditions. Most of the existing works on the Bose gas consider periodic boundary conditions, and focus on models of random permutations that dispense with considerations of boundary conditions. However, we find it important to keep them in the model, since they are physically relevant and since they might lead to additional effects. Indeed, for periodic or Dirichlet zero boundary conditions, the asymptotics of the weight $\Gamma^{\ssup{\rm bc}}_{\L,k}$ in \eqref{Gammabcdef} for large $k$ and boxes $\L$ of volume $\asymp N$ are clearly different for $k\ll N^{2/d}$ and for $k\gg N^{2/d}$.


\subsection{Our Purpose.}\label{sec-Purpose}

\noindent As mentioned above, we will be interested in the {\em thermodynamic limit} of the model described above. Namely, we fix the inverse temperature $\beta \in (0,\infty)$, a density $\rho\in (0,\infty)$, take the box $\L=\L_N$ of volume $|\L_N|=\frac N\rho$  
and define the corresponding {\it limiting free energy per volume} as 
\begin{equation}\label{freeenergymodified}
f(\beta,\rho):=-\lim_{N\to\infty}\frac 1{|\L_N|}\log  Z_N^{\ssup{\rm bc}}(\beta,\L_N),\qquad\beta,\rho\in(0,\infty).
\end{equation}
In our first main result, Theorem~\ref{thm-limfree},  we will see that  this limit exists, is independent of boundary conditions and can be expressed in terms of a {\em characteristic variational formula} that describes the statistics of all the lengths of the Brownian bridges. Next, in this result we prove a large-deviation principle for these statistics and identify the minimizer of the rate function. This implies in particular that the statistics of the cycle lengths converge in distribution towards that minimizer. Now, the characterization of the minimizer involves a possible phase transition in the density $\rho$ -- that is, a {\it non-analyticity} of the map 
\begin{equation}\label{non-analytic}
\rho\mapsto f(\beta,\rho)
\end{equation} emerges at some explicit {\em critical threshold} $\rho_{\rm c}(\beta)$, which may or may not be finite in general.  This phase transition underlines a {\em loss of mass} in the statistics of the finite-size cycle lengths in the spirit of the well-known effect in the free Bose gas; it is therefore a signal for a condensation phase transition in a vicinity of the Bose--Einstein condensation.  

So far, the above statements are valid for a wide choice of the weights $\Gamma_{\Lambda,k}^{\ssup{\rm bc}}$ in \eqref{cycleNew} as well as arbitrary $d\in \N$ and $\beta\in (0,\infty)$. Now for our model, $\rho_{\rm c}(\beta)$ is shown to be equal to the Green's function at the origin of a transformed self-interacting Brownian bridge, geometrically weighted with a crucial parameter, the radius of convergence. As a consequence, finiteness of the Green's function implies the existence of the aforementioned non-analyticity of the map \eqref{non-analytic} at $\rho_{\rm c}(\beta)$.

Theorem~\ref{thm-Gammak} is our second main result, where we determine circumstances under which the aforementioned phase transition does occur. Indeed, we show that for $d\geq 5$ and the interaction potential $v$ remaining sufficiently small, the critical density $\rho_c(\beta)$ is actually {\it finite}. The key idea here is the observation that the self-interacting Brownian bridge lies in the same universality class as  the famous {\em self-avoiding walk (SAW)}, the uniform distribution on $N$-step nearest-neighbour random walks on $\Z^d$ starting at zero and hitting no site twice. Indeed, since the pair functional $v$ is non-negative, the interaction $\sum_{1\leq i<j \leq k}V(B_{i},B_{j})$ repels each two legs from each other, in the spirit of a weak version of the SAW, the {\em self-repellent random walk}. One main difference to the present model of the interacting quantum Bose gas is that here it is {\em paths} that undergo a mutual repellence, not endpoints. 

The restrictions to $d\geq 5$ and small $\beta$ are due to a continuous version of the {\it lace expansion method} that we develop and employ in the proof of Theorem~\ref{thm-Gammak}. This part is a modification of the method developed in \cite{BHK}. We refer to Section~\ref{sec-laxeexp} for a short survey on this method, and to Section~\ref{sec-BEC-SAW} for precise conjectures about the behaviour of the SAW and models in its universality class. Surprisingly, these conjectures imply  the finiteness of the Green's function (i.e., the occurrence of the phase transition in our model) even in dimension $d=2$, not only in dimensions $d\geq 3$, being aligned with conjectures about the BEC phase transition. 

Let us remark that there is a recent similar investigation on a closely related model, the {\it interchange model}, which is conjectured to show a rather similar phase transition. Indeed, in \cite{ES22}, it is shown that ``infinite cycles'' appear in dimensions $d\geq 5$ if the inverse temperature is picked large enough. The similarities to our work are obvious, but our approach and methods are pretty different to the ones of \cite{ES22}.

\subsection{Main Results.}\label{sec-Results}

\noindent The formulation of our results will require setting up some further notation, which will subsequently be used in the entire sequel. As before, $\mathcal C$ will denote the space of continuous functions $[0,\beta]\to \R^d$. The canonical Brownian bridge measure is defined for any measurable $A\subset \Ccal$ as 
\begin{equation}\label{nnBBM}
\mu^{\ssup \beta}_{x,y}(A)=\frac{\P_x(B\in A;B_\beta\in\d y)}{\d y}.
\end{equation}
Its total mass is given by 
\begin{equation}\label{def-Gaussian}
\begin{aligned}
\mu_{x,y}^{\ssup \beta}(\Ccal) =\varphi_{\beta}(x,y) &:=\frac {\P_x(B_\beta\in\d y)}{\d y} =(2\pi\beta)^{-d/2}{\rm e}^{-\frac 1{2\beta}|x-y|^2}, 
\end{aligned}
\end{equation}
where $B=(B_t)_{t\in[0,\beta]}$ is a Brownian motion in $\R^d$ with generator $\frac 12\Delta$, starting from $x$ under $\P_x$. Hence, the total mass of $\mu_{0,0}^{\ssup{k\beta}}$ is equal to $(2\pi\beta k)^{-d/2}$.

We fix a bounded mesurable function $v\colon [0,\infty) \to [0,\infty)$ with compact support. Let us set 
\begin{equation}\label{Gammadef}
\begin{aligned}
\Gamma_k &=\mu_{0,0}^{\ssup{k\beta}}\Big[{\rm e}^{-\sum_{1\leq i<j\leq k}V(B_i,B_j)}\Big] \\
&=\E_{0}\Big[{\rm e}^{-\sum_{1\leq i<j\leq k}V(B_i,B_j)}\,\1\{B(k\beta)\in\d x\}\Big]\Big/\d x\Big|_{x=0},
\end{aligned}
\end{equation}
with $V(\cdot,\cdot)$ defined in \eqref{Vdef}, compare to \eqref{Gammabcdef}. Then the sequence $(\Gamma_k (2\pi\beta k)^{d/2})_{k\in \N}$ is sub-multiplicative. Consequently, by Fekete's lemma, the limit (known as the {\em connective constant}) 
\begin{equation}\label{lambdadef}
\begin{aligned}
&\lambda_{\rm c}(\beta)=\lim_{k\to\infty}\Gamma_k^{-1/k} \quad \mbox{exists and satisfies }\\
& \lambda_{\rm c}(\beta)^{k}\Gamma_k\geq (2\pi\beta k)^{-d/2} \quad\forall k\in \N.
\end{aligned}
\end{equation}
$\lambda_c(\beta)$ is the radius of convergence of the power series with coefficients $\Gamma_k$. It will turn out to capture the limiting behaviour of the coefficients $\Gamma_{\L_N,k}^{\ssup{\rm bc}}$ of the model. Finally, introduce 
\begin{equation}\label{Idef}
I(p)=\sum_{k\in\N}p_k\log\frac{p_k k}{\Gamma_k \e},\qquad p=(p_k)_{k\in\N}\in[0,\infty)^\N.
\end{equation}
The function $I$  will turn out to play the role of a large-deviation rate function -- it is the relative entropy of $p$ with respect to the sequence $(\Gamma_k/k)_{k\in\N}$ plus the sum of the $\Gamma_k/k$.
 
Here is the identification of the free energy $f(\beta,\rho)$ defined in \eqref{freeenergymodified}, which is the first main result of our article.

\begin{theorem}[Free energy and cycle lengths]\label{thm-limfree} Fix $\beta,\rho\in (0,\infty)$ and consider the model defined in \eqref{cycleNew} and \eqref{PNNewdef} in the centred box $\L_N$ with volume $N/\rho$ with Dirichlet boundary conditions ${\rm bc}\in \{ {\rm Dir}\}$. Then the following hold.
\begin{enumerate}
\item[(i)] With $\lambda_({\rm c}\beta)$ defined in \eqref{lambdadef}, the limiting free energy $f(\beta,\rho)$ defined in \eqref{freeenergymodified}
exists and is identified as
\begin{equation}\label{freeenergyident}
f(\beta,\rho)= \inf_{p\in[0,\infty)^\N\colon \sum_k kp_k\leq\rho}\Big[I(p)+\Big(\rho-\sum_{k\in\N}k p_k\Big)\log \lambda_{\rm c}(\beta)\Big].
\end{equation}

\item[(ii)] Let $(L_k)_{k\in\N}$ be a random variable under ${\tt P}^{\ssup N}_{\beta,\L_N}$, then the distribution of $(L_k/|\L_N|)_{k\in \N}$ satisfies a large-deviations principle (LDP) \footnote{We recall that a sequence $(X_N)_{N\in\N}$ of random variables taking values in a topological space $\mathcal X$ satisfies an LDP on the scale $\gamma_N$ with rate function $J\colon \mathcal X\to[0,\infty]$ if $J$ is lower-semicontinuous and 
for any open subset $G$ and any closed subset $F$ of $\mathcal X$,
$$
\liminf_{N\to\infty}\frac 1 {\gamma_N}\log\P(X_N\in G)\geq -\inf_G J\qquad\mbox{and}\qquad\limsup_{N\to\infty}\frac 1 {\gamma_N}\log\P(X_N\in F)\leq -\inf_F J.
$$
If $J$ has a unique minizer $x^*\in\mathcal X$, then it follows that $X_N$ converges weakly towards $x^*$. We refer to \cite{DZ98} for the general theory of large deviations.}
on the set $\mathcal X_\rho=\{p\in[0,\infty)^{\N}\colon\sum_{k\in\N}kp_k\in[0,\rho]\}$ on the scale $|\L_N|$ with rate function $J$ given by
\begin{equation}\label{Jdef}
J(p)=I(p)+\Big(\rho-\sum_{k\in\N}k p_k\Big)\log \lambda_{\rm c}(\beta)- f(\beta,\rho).
\end{equation}

\item[(iii)] Put   
\begin{equation}\label{critdens}
\rho_{\rm c}(\beta)=\sum_{k\in\N}\lambda_{\rm c}(\beta)^{k}\Gamma_k\in[0,\infty].
\end{equation}
Then the unique minimizer $p^*$ of the formula on the right-hand side of \eqref{freeenergyident} is given as 
\begin{equation}\label{pkELequation}
p^*_k=\frac 1k \lambda_{\rm c}(\beta)^{k}\Gamma_k \times\begin{cases} \e^{-c (\rho) k},&\mbox{if }\rho\leq \rho_{\rm c}(\beta),\\
1&\mbox{if } \rho>\rho_{\rm c}(\beta),
\end{cases}\qquad k\in\N,
\end{equation}
where $c(\rho)\in[0,\infty)$ is such that 
$$
\sum_{k\in\N}\lambda_{\rm c}(\beta)^{k}\Gamma_k \e^{-c (\rho) k}=\rho.
$$
 In particular, $\rho\mapsto c(\rho)$ is strictly decreasing with $c(\rho_{\rm c}(\beta))=0$. Furthermore, $\sum_{k\in\N}k p^*_k=\rho\wedge\rho_{\rm c}(\beta)$ for any $\rho\in(0,\infty)$. 

\item[(iv)]
\begin{equation}\label{freeeenergyident2}
 f(\beta,\rho)=\rho \log\lambda_{\rm c}(\beta)-\begin{cases}\rho c(\rho)+\sum_{k\in\N}\frac 1k \Gamma_k\lambda_{\rm c}(\beta)^k\e^{-c(\rho) k}&\mbox{if }\rho\leq \rho_{\rm c}(\beta),\\
\sum_{k\in\N}\frac 1k \Gamma_k\lambda_{\rm c}(\beta)^k&\mbox{if }\rho>\rho_{\rm c}(\beta). 
\end{cases}
\end{equation}
In particular, $f(\beta,\cdot)$ has a phase transition (non-analyticity) in $\rho_{\rm c}(\beta)$ if this point is finite.
\end{enumerate}
\end{theorem}

The proof of Theorem~\ref{thm-limfree} is provided in Section~\ref{sec-freeenerproof}. 
Let us make some remarks about the assertions appearing there. First note that the rate function of the LDP there has two terms, the entropic term $I$, which describes the statistics of the finite cycle lengths, and the energy term, $(\rho-\sum_k kp_k)\log\lambda_{\rm c}(\beta)$, which is the contribution from the condensate part. Unlike in the free Bose gas (see, e.g., \cite{KVZ23}), the condensate yields here an explicit contribution on the leading scale, but not an entropic one. While each cycle of length $k$ comes with an energetic term $\Gamma_k$, the condensate comes with the logarithm of the connective constant of the sequence $(\Gamma_k)_{k\in\N}$. The convergence of $(L_k/\L_N)_{k\in\N}$ towards the minimizer $p^*$ shows that the model has a phase transition at the critical density $\rho_{\rm c}(\beta)$ if this quantity is finite. In \eqref{critdens} we see that  $\rho_{\rm c}(\beta)$
is the Green's function of the transformed Brownian motion geometrically weighted with parameter $\lambda=\lambda_{\rm c}(\beta)$, taken at the origin, and $\lambda_{\rm c}(\beta)$ is the radius of convergence. If the Green's function is finite at this point, then the free energy $f(\beta,\cdot)$ is not analytic here, since all the coefficients in the series in \eqref{freeenergyident} are positive. 

Now the question of course arises, under what circumstances this phase transition occurs. We give a positive answer in our next main result.

\begin{theorem}[Phase transition in $d\geq 5$]
\label{thm-Gammak}
Assume that $d\geq 5$ and that the interaction potential $v$ (cf.\ \eqref{Gammabcdef} and \eqref{Vdef}) is bounded and continuous and has a bounded support. Let $\rho_{\rm c}(v, \beta)=\rho_{\rm c}(\beta)$ be the critical density defined in \eqref{critdens}. Then, for any $\beta\in(0,\infty)$, there is $\be_\beta>0$ such that $\rho_{\rm c}(\be v,\beta)<\infty$ for any $\be\in(0,\be_\beta]$. 
\end{theorem}

The proof of Theorem~\ref{thm-Gammak} is spanned through Section~\ref{sec-ProofBEC} -- Section \ref{sec-technical-est}. As we will explain in Section~\ref{sec-background}, the restriction to $d\geq 5$ 
is related to the 
question that lies at the heart of a famous and fundamental question that is notoriously difficult to answer in 
dimensions $d\in\{2,3,4\}$: the question whether or not the Green's function of the weakly self-avoiding walk is finite at the critical point.

The organization of the remainder of this paper is as follows. In Section~\ref{sec-background} we discuss several aspects of our main results, in Section~\ref{sec-freeenerproof} we prove Theorem~\ref{thm-limfree}, in Section~\ref{sec-ProofBEC} we give the proof of Theorem~\ref{thm-Gammak}, subject to the proof of two crucial results that we prove in Sections~\ref{Le_1_proof} and \ref{Le_2_proof}, respectively. In Section~\ref{sec-technical-est}, we prove some technical estimates, which are used in the two sections preceding to that. 


\section{Background and Discussion}\label{sec-background}

In this section we discuss the background, proof method and conclusions of our main results. More concretely, in Section~\ref{sec-BoseModel} we explain the relation between the well-known interacting Bose gas and the model that we study in this paper, in Section~\ref{sec-BEC-SAW} we summarize what is known and what is conjectured about the weakly self-avoiding walk, and we draw conjectural conclusions about the occurrence of the phase transition in those dimensions that we do not handle in this paper. Some elements and outline of our main proof method for Theorem~\ref{thm-Gammak} is presented and commented on in Section~\ref{sec-laxeexp}. Finally, we comment on our model from the viewpoint of random permutations in Section~\ref{sec-RandPerm}, and in Section~\ref{sec-bc} we comment on the role of the boundary conditions carried in our model.

\subsection{The Interacting Bose Gas.}\label{sec-BoseModel}

\noindent The model that we are interested in and described in Section \ref{sec-model} is strongly inspired by 
an investigation of an interacting quantum gas at positive temperature $1/\beta \in (0,\infty)$ in the thermodynamic limit, the setup for which
can be briefly described as follows. As before, let $\Lambda\subset \R^d$ be a centered box and 
$$
\Hcal^{\ssup{\rm {bc}}}_{N,\L}=-\frac 12\sum_{i=1}^N\Delta_i+\sum_{1\leq i<j\leq N} v(|x_i-x_j|),\qquad x_1,\dots,x_N\in\Lambda, 
$$
is the Hamiltonian operator for $N$ particles located at $x_1,\dots,x_N$ in $\L\subset\R^d$ with a prescribed boundary condition.  
Again $v\colon [0,\infty)\to[0,\infty)$ is some pair-interaction functional, which is assumed to be bounded with bounded support.  We are interested in {\it bosons}, and in this vein a key object of interest is the {\em symmetrized trace} 
\begin{equation}\label{defpartition}
\mathbf Z_N^{\ssup{\rm bc}}(\beta,\L):={\rm Tr}_{+}\big(\e^{-\beta\Hcal^{\ssup{\rm {bc}}}_{N,\L}}\big), \qquad \beta\in (0,\infty),
\end{equation}
where the index $+$ stands for symmetrization -- that is, application of the projection operator on the set of all permutation invariant wave functions. 

We are going to rewrite this trace in terms of many interacting Brownian bridges, which is a well-known manipulation since the early 1970s and builds on Feynman's intuition from 1953.
We refer to \cite[Lemma~2.1]{ACK10} for a proof and further details for the following, and to \cite{U06b} for a non-technical and physics-related introduction to the rewrite of the Bose gas in terms of  Brownian bridges. 

Using well-known trace formulas \cite{G70}, \eqref{defpartition} can be rewritten in probabilistic terminology using {\it Brownian bridges} as 
\begin{equation}\label{bfZNBrowBrid}
\mathbf Z_N^{\ssup{\rm bc}}(\beta,\L)=\frac{1}{N!}\sum_{\sigma\in\Sym_N}\int_{\L^N}\d x_1\cdots\d x_N\bigotimes_{i=1}^N {\mu_{x_i,x_{\sigma(i)}}^{\ssup{\rm bc,\beta}}}\Big[\exp\Big\{-\sum_{1\le i<j\le N}V(B^{\ssup{i}},B^{\ssup{j}})\Big\}\Big].
\end{equation}
Next, it is well-known that each permutation $\sigma\in \Sym_N$ can be decomposed into cycles -- that is, each $\sigma$ can be represented by a partition of $N$. 
For each $\sigma$ there is an $l\in{\mathfrak P}_N$ such that $\sigma$ consists of precisely $l_k$ cycles of length $k$, for any $k$. Using the Markov property of the Brownian motion (or, equivalently, the semigroup property  of the family of Brownian bridge measures), we can identify each sequence of legs that lie within one cycle of length $k$ as one Brownian bridge starting and ending at the same site and having the time interval $[0,\beta k]$. Hence, for any $k\in \N$, $\sigma$ gives rise to $l_k$ bridges of length $k$, labeled 
\begin{equation}\label{def-B-klk}
B^{\ssup{k, 1}},\dots,B^{\ssup{k, l_k}},
\end{equation}
and any bridge $B^{\ssup{k,i}}$ has the legs
\begin{equation}
B^{\ssup{k,i}}_j=(B_j^{\ssup{k,i}}(s))_{s\in[0,\beta]}=(B^{\ssup{k,i}}((j-1)\beta+s))_{s\in[0,\beta]}.
\end{equation}
Then, writing $[m]=\{1,\dots,m\}$, our ensemble of legs is the collection 
$$
\bigg\{(B_j^{\ssup{k,i}})_{k,i,j}\colon k\in [N],\,\, i\in[l_k],\,\,j\in\ [k]\bigg\}
$$
sampled under the measure $\bigotimes_{k\in\N}(|\Lambda|\mu_{\Lambda}^{\ssup{{\rm bc},k\beta}})^{\otimes l_k}$. Furthermore, observe that, for any $l\in\mathfrak P_N$, the number of permutations of $1,\dots,N$  such that $l_k$ is the number of its cycles of length $k$ for any $k\in\N$ is equal to $N!/\prod_{k\in\N} l_k! \,k^{l_k}$ \cite[Th.~12.1]{C02}. As a result, the partition function of the interacting Bose gas may be written
\begin{equation}\label{partfctBose}
\sum_{l\in\mathfrak{P}_N}  \Big(\prod_{k\in\N}\frac{|\L|^{l_k}}{l_k!\,k^{l_k}}\Big)\Big[\bigotimes_{k\in\N}\big(\mu_{\L}^{\ssup{{\rm bc},k\beta}}\big)^{\otimes l_k}\Big]\big[{\rm e}^{-{G}_{N,\beta}}\big],
\end{equation}
where  the entire interaction can be written as 
\begin{equation}\label{oldinteraction}
{G}_{N,\beta}= \sum_{(k_1, i_1, j_1)\neq (k_2, i_2, j_2)} V\big(B^{\ssup{k_1, i_1}}_{j_1},B^{\ssup{k_2, i_2}}_{j_2}\big).
\end{equation}
Hence, in the standard form of the interacting Bose gas, every leg of every bridge interacts with any other leg of a bridge. If we drop all interactions between two legs of different bridges, that is, if we replace ${G}_{N,\beta}$ on the right-hand side of \eqref{partfctBose} by the interaction
\begin{equation}\label{GWNewdef}
H_{N,\beta}=\sum_{k =1}^N \sum_{i =1}^{l_{k}}\sum_{1\leq j_1<j_2\leq k} V\big(B^{\ssup{k, i}}_{j_1},B^{\ssup{k, i}}_{j_2}\big), 
\end{equation}
then, for this interaction, the expectation over the product of the $\mu_{\L}^{\ssup{{\rm bc},k\beta}}$ decomposes into a product of expectations, and we see that in \eqref{partfctBose}, the partition function of our model, $Z_N^{\ssup{\rm bc}}(\beta,\L)$, defined in \eqref{cycleNew} arises. This is our motivation to study this model. Although different cycles do not interact and it is a simplification compared to the scenario of the interacting Bose gas, the self-interactions within each cycle already manifest into an interesting phase transition, as highlighted by Theorem \ref{thm-limfree} and Theorem \ref{thm-Gammak},  confirming the conjecture that we will now be discussing in 
Section~\ref{sec-BEC-SAW}.

\subsection{Self-repellent Random Motions.}\label{sec-BEC-SAW}

\noindent In our model \eqref{cycleNew} of a partially interacting Bose gas, it turned out in Theorem~\ref{thm-limfree} that the occurrence of a condensation phase transition is equivalent to the summability of $\lambda_{\rm c}(\beta)^k\Gamma_k$ on $k\in\N$. This in turn is equivalent to the finiteness of the Green's function at zero for the corresponding self-repellent random path model in the critical point. This model is of the type of the weakly self-avoiding walk, however, for a Brownian bridge instead of a free (i.e., open-end) path, and with a different type of mutually repellent interaction, which has not yet been studied in the literature on self-interacting path models. Indeed, the interaction $\sum_{1\leq i<j\leq k}V(B_i,B_j)$ for a Brownian motion $B$ is in the same spirit as the interaction $H_k=\sum_{1\leq i<j\leq k}\1\{S_i=S_j\}$ (the {\em self-intersection local time}) for a random walk $(S_i)_{i\in\N_0}$ in $\Z^d$, since $V>0$. The two  differences between these two models are that (1) precise intersections do almost surely not appear for the Brownian-motion model, and (2) the entire leg $B_i\colon[0,\beta]\to\R^d$ is involved in the interaction, not only its endpoint $B_i(\beta)$. But we find it obvious that both models should lie in the same universality class of transformed path measures. Hence let us briefly summarize what is known for the weakly self-avoiding walk.

See \cite{S11} for a brief survey on the weakly self-avoiding walk, the random walk measure on $k$-step nearest-neighbour paths $(S_0,S_1,\dots,S_k)$ in $\Z^d$ starting from the origin with probability weight $\exp\{-\be \sum_{1\leq i<j\leq k}\1\{S_i=S_j\}\}$ and partition function $Z_k$. We restrict to dimensions $d\geq 2$. The main conjecture is that the expected end-to end distance $|S_k-S_0|$ after $k$ steps should be of size $\approx k^{\nu_d}$, with $\nu_d\in[\frac 12,1]$ a critical power that is believed to be $\nu_2=\frac 34$, $\nu_3\approx 0.5876$, and $\nu_d=\frac 12$ in $d\geq 4$, with logarithmic corrections in $d=4$. 
Furthermore, $Z_k$ is  conjectured to be $\asymp \lambda^{-k} k^{\gamma_d -1}$ for some $\gamma_d$. The powers $\nu_d$ and $\gamma_d$ are believed to be universal, i.e., not to depend on details of the underlying random walk. Their conjectured values are $\gamma_2=\frac{43}{32}$ and $\gamma_3\approx 1.1568$ and $\gamma_d=1$ for $d\geq 4$ with logarithmic corrections in $d=4$. This implies that the partition  function in the critical point should satisfy $Z_k\approx \lambda_{\rm c}^{-k} k^{\gamma_d-1}$ as $k\to\infty$, where $\lambda_{\rm c}$ is the critical parameter (the connective constant, defined analogously to \eqref{lambdadef}).

Conjectures about the finiteness of the Green's function in the critical point can be deduced from the above conjectures. Recall that the prediction is that $|S_k|\approx k^{\nu_d}$ and $Z_k\approx \lambda^{-k} k^{\gamma_d-1}$ as $k\to\infty$, where we remark that $Z_k$ is the experctation of $\e^{-\be H_k}$ for the random walk with a free end, while $\Gamma_k$ is the same for the random walk bridge. Assuming even that $S_k k^{-\nu_d}$ converges in distribution under the (free-end) weakly self-avoiding walk model towards some non-degenerate variable (and neglecting the logarithmic corrections that are expected in $d=4$), we would then obtain that
$$
\Gamma_k \lambda^k\approx k^{\gamma_d-1}\frac 1{Z_k}\E\big[\e^{- \be H_k}\1_{\{S_k k^{-\nu_d}\in\d x\}}\big]/\d x|_{x=0}\approx k^{\gamma_d-1-d\nu_d}.
$$
That is, if the right-hand side is summable, then the model with bridge instead of free end should have a finite Green's function in the critical point. This is the case precisely if  $\gamma_d<d\nu_d$. Surprisingly, this should be true in any dimension $d\in\{2,3,4\}$, also for $d=2$. 

Hence, we conjecture that the condensation phase transition appears in our model in \eqref{cycleNew} also in $d=2$. This is surprising, since in the original interacting quantum Bose gas (see Section~\ref{sec-BoseModel}) it is generally conjectured that the phase transition exists only in dimensions $d\geq 3$, but not in $d=2$. For the two-dimensional free (i.e., non-interacting) quantum Bose gas, the non-existence of the phase transition relies on the non-summability of $1/k$, which is \lq just at the boundary\rq\ of the validity of this statement.

\subsection{Lace Expansion.}\label{sec-laxeexp}

\noindent Let us comment on the method that we use in the proof of Theorem~\ref{thm-Gammak}, the {\em lace expansion}. This is a diagrammatic expansion of the exponential interaction term $\e^{-\be \sum_{1\leq i<j\leq k}\1\{S_i=S_j\}}$ in terms of a procedure that reminds of the inclusion-exclusion principle from basic probability theory. Since diagrammatic illustrations of the method remind of the laces on the edge of a tablecloth, this method got its name ``lace expansion''. It was introduced in the mid-eighties by Brydges-Spencer \cite{BS85} for weakly self-avoiding walks in dimension $d>4$ , proving a central limit theorem, if the interaction parameter
$\alpha>0$ is small enough. The method was later extend in many works to different fields,
for instance to percolation in sufficiently high dimensions see \cite{vdHS02,HvdHS03}, and also in \cite{HS92} to strictly self-avoiding
random walks in dimensions $5$ and above, see also the recent article \cite{S22}.

The method is perturbative and needs a simple, but non-trivial model to expand around, which is most often the simple random walk. Since (at least conjecturally) the weakly self-avoiding walk is similar to the simple random walk only in dimensions $d\geq 5$, the method has been successful only in these dimensions, leading to rather strong assertions, among which are the following: the endpoint satisfies a central limit theorem, 
i.e., the distribution of $S_k/\sqrt k$ under the transformed measure converges towards a Gaussian distribution with an appropriate non-trivial variance. Furthermore, the Green's function is finite in the critical parameter value $\lambda_{\rm c}$.

Besides many other extensions, in recent years the lace expansion was extended to models in continuous space, where self-intersections almost surely do not occur. In \cite{ABR16}, which gave us the leading idea for the proof in the present paper, the following model is successfully handled in this vein. Here the energy $V(f,g)$ defined in \eqref{Vdef} is replaced by 
\begin{equation}\label{VSAWdef}
V_{\rm SAW}(f,g)=\be\1\{|f(\beta)-g(\beta)|\leq r\},\qquad f,g\in\Ccal,
\end{equation}
where $r,\be>0$ are parameters, with a free end of the motion. Then the model is the weakly self-avoiding walk for a Gaussian random walk on $\R^d$, whose steps have mean zero and variance $\beta$. In \cite{ABR16}, it was proved that some result in the vicinity of the local central limit theorem holds:
\begin{equation}\label{LCLT}
\lim_{k\to\infty}\frac 1{Z_k} \E_0\Big[\e^{-\be \sum_{1\leq i<j\leq k}V(S_i,S_j)}\,\1\{k^{-1/2}S_k \approx x\}\Big]=g_D(0,x),\qquad x\in \R^d \setminus U,
\end{equation}
for sufficiently small $\be>0$ and some neighbourhood $U$ of $0$, where $Z_k$ is the partition function, $g_D$ is the Gaussian transition kernel with $D=D_\be>0$ a suitable diffusion constant, and ``$\approx$'' needs to be adjusted to a proper transition from discrete to continuous space.
Along the proof, it was also shown that $Z_k\sim C\lambda^{-k}$ for some $C,\lambda>0$, implying of course that $\lambda=\lambda_{\rm c}(\gamma)$ is the critical parameter.  For $x=0$, we have only ``$\leq$'' in \eqref{LCLT}, which says that $\Gamma_k\leq O(\lambda_k^{-k} k^{-d/2})$, which implies that $\Gamma_k\lambda^k$ is summable, i.e., the Green's function in the critical point is finite. This implies that Theorem~\ref{thm-Gammak} holds for this model for sufficiently small interaction parameter $\be$. 

The novelty of the present paper is to extend it to the interaction $V$ defined in \eqref{Vdef}. This is a substantial non-trivial step 
since the interaction depends now on the {\it entire leg} $B_i$, not only on its endpoint $B_i(\beta)$. 
The new ideas needed in the present context, for which we will draw on the lace expansion method using Banach algebras introduced in \cite{BHK}, will be discussed along the way of proving Theorem \ref{thm-Gammak}.

\subsection{Random Permutations.}\label{sec-RandPerm}

\noindent The model that we study in this paper is a probability measure on the set $\mathfrak P_N$ of all integer partitions of $N$ of the form 
\begin{equation}\label{measuredef}
{\tt P}_N(l)=\frac1{Z_N}\prod_{k\in\N}\frac{[\theta^{\ssup N}_k]^{l_k}}{l_k! \,k^{l_k}},\qquad l=(l_k)_{k\in\N}\in\mathfrak{P}_N.
\end{equation}
Here $Z_N$ is the normalisation constant that turns ${\tt P}_N$ into a probability measure, and $\theta_k^{\ssup N}\in(0,\infty)$ are positive weights. Such measures are also known as models of {\em random permutations}, i.e., as a probability measure on the set $\mathfrak S_N$ of all permutations $\sigma$ of $1,\dots,N$, giving the weight $\theta_k^{\ssup N}$ to any cycle of $\sigma$ with length $k$, independently over all cycles.  Note that one can freely multiply $\theta_k^{\ssup N}$ by a factor $\e^{\alpha k}$ for fixed $k$ without changing ${\tt P}_N$, since this factor drops out in the normalisation.

For $\theta_k^{\ssup N}=\theta_k$ not depending on $N$, such distributions have been studied in recent years for many choices of the weights $\theta_k$, and various limiting regimes and limiting distributions have been identified and studied in great detail, producing a rich picture. This line of research goes under the name  {\em random (spatial resp.~geometric) permutations} and was initiated in \cite{BUe09} as a simplified model for the interacting Bose gas (it appears after interchanging the logarithm with a certain sum in the partition function). See \cite{RZ20} for one of the latest contributions and a host of references on this line of research, which has departed considerably from the spirit of the interacting Bose gas in the last decade.

The weights that we study in this paper, are roughly of the form
\begin{equation}\label{thetachoice}
\theta_k^{\ssup N}= N (\gamma_{N,k}+o(1)),\qquad N\to\infty, k\in\N,
\end{equation}
with some $\gamma_{N,k}\in(0,\infty)$, whose behaviour as $N\to\infty$ is {\it a priori}  not easy to determine. Indeed, for fixed $k$, it converges, as $N\to\infty$ towards $\gamma_k=\frac 1\rho \Gamma_k$ defined in \eqref{Gammadef}, and this is also true for $k\ll N^{2/d}$, but for $k\gg N^{2/d}$, other effects emerge in the asymptotics, and a proof of the necessary asymptotics for deriving an LDP are not so easy to get, and actually in the case of periodic boundary condition we do not. Note that $\Gamma_k$ is exponentially fast in $k$, but this exponential terms drops out of the question about the phase transition, and the second term is decisive.  For the free Bose gas, $\gamma_k=\frac 1\rho (2\pi\beta k)^{-d/2}$ with $\rho,\beta\in(0,\infty)$ are the particle density and inverse temperature.

\subsection{Influence of Boundary Conditions.}\label{sec-bc}

\noindent The technical difficulties in deriving an LDP in Theorem \ref{thm-limfree} stem from the precise form of $\gamma_{N,k}$ in \eqref{thetachoice}, which depend on the boundary conditions that we use in our interacting Bose gas model. 

The by far most often used condition (in the standard interacting Bose gas and in practically all variants) is the periodic one, which is a mathematical idealisation that gives simpler formulas and less technicalities. Many papers on the free Bose gas (e.g. \cite{S93, S02, BCMP05}) concentrate either on this choice or on the even simpler {\it free boundary conditions}, where the starting and termination sites of the bridges stay inside the box, but not necessarily the entire path (this does not come from a quantum mechanical model). However, for zero boundary condition,  \eqref{thetachoice} is actually not satisfied in the free Bose gas for $k$ coupled with $N$, since for $k\gg N^{2/d}$, the probability to stay in the box with diameter $\asymp N^{1/d}$ decays to zero exponentially in $kN^{-2/d}$. To the best of our knowledge, this seems to be a fact that has not been noticed nor handled in the study of the free Bose gas before 2023, even though this case is highly relevant from the physics point of view. In the recent paper \cite{KVZ23}, this gap is closed by deriving asymptotics for the reduced one-particle matrix of the free Bose gas for all these well-known boundary conditions, also including Neumann conditions, to prove off-diagonal long-range order, which is considered to be the equivalent to the occurrence of BEC.

In our present paper, however, the technicalities arising in the proof of the LDP turned out serious for the case of periodic boundary condition and we decided to complete the proof only for Dirichlet boundary condition. Indeed, we were not able to prove the upper bound $\Gamma_{\L_N,k}^{\ssup{\rm per}}\leq \lambda_{\rm c}(\beta)^k(1+\eps)^k$ for any $k\in\{1,\dots,N\}$ and all sufficiently large $N$, precisely for the reason mentioned above (the asymptotics for $k\gg N ^{2/d}$ are difficult to handle). We elaborate on this point in Remark \ref{remark periodic} towards the end of Section~\ref{sec-ApprGamma}.

\section{Proof of Theorem~\ref{thm-limfree}} \label{sec-freeenerproof}
The proof of Theorem \ref{thm-limfree} is split into three steps: In Section \ref{sec-ApprGamma} we will show that the limiting free energy exists as $N\to\infty$ 
and is independent of the Dirichlet boundary conditions imposed.  The assertions (i) and (ii) of Theorem \ref{thm-limfree} will be shown in Section \ref{sec-LDPproof},
while assertions (iii) and (iv) will be proved in Section \ref{sec-Varformproof}.

\subsection{Approximation of $\Gamma^{\ssup{\rm bc}}_{\L_N,k}$ with $\Gamma_k$.}\label{sec-ApprGamma}

Recall from \eqref{Gammadef} and from \eqref{Gammabcdef} that 
\begin{eqnarray}
\Gamma_k &=&\E_{0}\Big[{\rm e}^{-H_k(B)}\,\1\{B(k\beta)\in\d x\}\Big]\Big/\d x\Big|_{x=0},\\
\Gamma^{\ssup{\rm Dir}}_{\L_N,k}&=&\frac 1 {|\Lambda_N|} \int_{\Lambda_N} \d x\, \E_x \Big[\e^{-H_k(B)}\1\{B(s)\in\L_N\,\forall s\in[0,\beta k]\}\1\{B(k\beta)\in\d x\}\Big]\Big/\d x,\label{GammaDir}
\end{eqnarray}
where we put 
$$
H_k(B)=\sum_{1\leq i< j\leq k} V(B_i, B_j) =\sum_{1\leq i< j\leq k} \int_0^\beta\d s\, v\big(|B_i(s)-B_j(s)|\big)
$$
and recall that $B_i=(B((i-1)+s))_{s\in[0,\beta]}$ is the $i$-th leg of the Brownian motion $B$ on $[0,k\beta]$.

In our proof of the LDP in Section~\ref{sec-LDPproof} below, we will need upper bounds for $\Gamma_{\L_N,k}^{\ssup{\rm Dir}}$ in terms of $\Gamma_k$ for any $k\in\{1,\dots,N\}$, but lower bounds only for any $k\leq k_N$ for some $k_N\to\infty$. In this section, we formulate and prove these assertions in \eqref{simpleesti} and \eqref{Gammalowbound}. 

For convenience, we put $\L_N=[-\frac 12 L_N,\frac 12 L_N)^d$. For $x\in\R^d$, we define $z(x)\in\Z^d$ by $x\in z(x)L_N$ and $[x]=x-z(x)L_N\in\L_N$. 

First we note that
\begin{equation}\label{simpleesti}
\Gamma_k\geq \Gamma^{\ssup{\rm Dir}}_{\L_N,k},\qquad N,k\in\N, 
\end{equation} 
which follows from dropping the indicator on the event $\{B(s)\in\L_N\,\forall s\in[0,\beta k]\}$
in \eqref{GammaDir} and noting the shift-invariance of $H_k(B)$.

We now turn to a lower bound for $\Gamma^{\ssup{\rm Dir}}_{\L_N,k}$ for $k$'s that are fixed or diverging slowly. Pick some integer $k_N\to\infty$ such that $k_N\leq o(N^{1/d})$. Then we claim that for any $\eps\in(0,1)$, there is $N_0$ such that
\begin{equation}\label{Gammalowbound}
\Gamma_{\L_N,k}^{\ssup{\rm Dir}}\geq \Gamma_{k}(1-\eps) ,\qquad N\geq N_0, k\in\{1,\dots,k_N\}.
\end{equation}
Indeed, with a small $\eta\in(0,1)$, split the integration in \eqref{GammaDir} into $(1-\eta)\L_N$ and $\L_N\setminus (1-\eta)\L_N$ to see that
$$
\begin{aligned}
\Gamma_{\L_N,k}^{\ssup{\rm Dir}}
&\geq \Gamma_{k}
-\frac 1{|\L_N|}\int_{(1-\eta)\L_N}\d x\,\mu_{x,x}^{\ssup{\beta k}}\Big[\e^{-H_k(B)}\1\{B(s)\in \L_N^{\rm c}\text{ for some }s\in[0,\beta k]\}\Big]\\
&\qquad -\frac 1{|\L_N|}\int_{ \L_N\setminus (1-\eta)\L_N}\d x\,\mu_{x,x}^{\ssup{\beta k}}\Big[\e^{-H_k(B)}\Big]\\
&\geq \Gamma_{k}-\e^{- C N^{2/d}/k}-\frac{|\L_N\setminus (1-\eta)\L_N|}{|\L_N|} \Gamma_{k} 
\\
&\geq \Gamma_{k}- \e^{- C N^{2/d}/k}-C\eta\Gamma_{k},
\end{aligned}
$$
where we first used an estimate for the probability that a Brownian bridge on the time interval $[0,k \beta]$ reaches a site that is $\eta N^{1/d}$ away (with some suitable $C>0$, not depending on $N$) and then used the maximal growth of $k$.  Now use that $\Gamma_{k}=\lambda_{\rm c}(\beta)^{-k(1+o(1))}$ has a nontrivial exponential behaviour in $k$ to see that the latter bound can be lower estimated against $\Gamma_{k}(1-\eps)$ for any $N\geq N_0$, if $\eta$ and $N_0$ are picked suitably. This proves 
\eqref{Gammalowbound}.

\begin{remark}\label{remark periodic}
    For the periodic boundary conditions, we are not able to derive the corresponding bounds for the following reason: Indeed, by shift-invariance of the interaction
$$
H_{\L_N,k}^{\ssup {\rm per}}(B)=\sum_{z\in\Z^d}\sum_{1\leq i<j\leq k}\int_0^\beta v\big(|B_i(s)-B_j(s)+z L_N|\big)\,\d s,
$$
we see that
$$
\Gamma_{\L_N,k}^{\ssup{\rm per}}=\sum_{z\in\Z^d} \E_0\Big[\e^{- H_{\L_N,k}^{\ssup {\rm per}}(B)}\1\{B(k\beta)\in \d y\}\Big]\Big|_{y=z L_N}.
$$
Now we can invoke the lower bound $\Gamma_{\L,k}^{\ssup{\rm per}}\geq \Gamma_{\L,k}^{\ssup{\rm Dir}}$ and use our lower-bound proof for the latter. We can also estimate $H_{\L_N,k}^{\ssup {\rm per}}(B)\geq H_k(B)$. However, the problem that we encounter now is to upper-estimate the expectation density on $\{B(k\beta)\in L_N \Z^d\}$ against the one on $\{B(k\beta)=0\}$. In fact, this seems to be a rather deep issue, since the behavior of the self-repellent walk depends on the spread of the path in a complicated manner that is currently not understood. However, under the assumptions that we imposed in Theorem~\ref{thm-Gammak}, an extension of that proof might very well be sufficient to derive suitable asymptotics, but we have refrained from pursuing this here.
\end{remark}

\subsection{Proof of the LDP.}\label{sec-LDPproof}

\noindent In this section we will prove assertions (i) and (ii) of Theorem~\ref{thm-limfree}. Recall the set of partitions,  $\mathfrak P_N=\{l=(l_k)_{k\in \N}\in \N_0^\N\colon \sum_k k l_k=N\}$ and the state space $\mathcal X_\rho=\{p\in [0,\infty)^{\N}\colon \sum_k kp_k\in[0,\rho]\}$. We are going to prove that, for any $p\in \mathcal X_\rho$,
\begin{equation}\label{LDPball}
\lim_{\delta\downarrow 0}\lim_{N\to\infty}\frac 1{|\L_N|}\log {\tt P}^{\ssup N}_{\beta,\L_N}\Big((L_k/|\L_N|)_{k\in\N}\in B_\delta(p)\Big)=-J(p),
\end{equation}
where the rate function $J(p)$ is introduced in \eqref{Jdef}, and $B_\delta(p)$ is the $\delta$-ball around $p$ in some metric that induces the product topology on $[0,\infty)^{\N}$. Since the state space $\mathcal X_\rho$ is compact, from \eqref{LDPball} the LDP in (ii) and the identification of the free energy in (i) of Theorem~\ref{thm-limfree} follow along some standard arguments from LDP theory.

Fix $p\in \mathcal X_\rho$.  From a combinatorial argument \cite{A76} it follows that $\#\mathfrak P_N =\e^{o(N)}$ as  $N\to\infty$. Hence, the upper bound  in   \eqref{LDPball} will follow from 
\begin{equation}\label{LDPproof}
\begin{aligned}
&\limsup_{\delta\downarrow 0}\limsup_{N\to\infty}\frac1{|\L_N|}\log \bigg(\sup_{l\in\mathfrak P_N\colon l/|\L_N|\in B_\delta(p)} \prod_{k\in \N}\frac{\big[|\L_N|\Gamma_{\L_N,k}^{\ssup{\rm bc}}\big]^{l_k}}{l_k! k^{l_k}}\bigg)
\\
&\leq -I(p)-\Big(\rho-\sum_{k\in\N}kp_k\Big)\log\lambda_{\rm c}(\beta),
\end{aligned}
\end{equation}
and the lower bound will follow from 
\begin{equation}\label{LDPprooflower}
\begin{aligned}
&\liminf_{\delta\downarrow 0}\liminf_{N\to\infty}\frac1{|\L_N|}\log \bigg(\prod_{k\in \N}\frac{\big[|\L_N|\Gamma_{\L_N,k}^{\ssup{\rm bc}}\big]^{l^{\ssup{\delta,N}}_k}}{l^{\ssup{\delta,N}}_k! k^{l_k^{\ssup{\delta,N}}}}\bigg)
&\geq -I(p)-\Big(\rho-\sum_{k\in\N}kp_k\Big)\log\lambda_{\rm c}(\beta),
\end{aligned}
\end{equation}
for some $l^{\ssup{\delta,N}}\in B_\delta(p)$. Then combining \eqref{LDPproof} and \eqref{LDPprooflower} will imply \eqref{LDPball}. 

\noindent{\bf Proof of \eqref{LDPproof}:} Pick $\delta>0$ and  $l\in\mathfrak P_N$ such that $l/|\L_N|\in B_\delta(p)$. We need a large auxiliary parameter $L\in\N$. We split the product on $k$ on the left-hand side of \eqref{LDPproof} into the products on $k\leq L$ and on $k>L$.  We substitute
\begin{equation}\label{pNLdef}
p_k^{\ssup{N,L}}=\frac {l_k}{|\L_N|}, \qquad k\in \{1,\dots,L\}. 
\end{equation}
Using Stirling's formula and \eqref{simpleesti}-\eqref{Gammalowbound}, the first partial product can be asymptotically identified as
\begin{equation}\label{mainpart}
\begin{aligned}
\prod_{k=1}^L \Big[\frac 1{l_k!}\Big(\frac 1k|\L_N|\Gamma_{\L_n,k}^{\ssup{\rm bc}}\Big)^{l_k}\Big]
&=
\e^{o(|\L_N|)} \prod_{k=1}^L\frac{(\Gamma_k/k)^{p^{\ssup{N,L}}_k |\L_N|}}{(p^{\ssup{N,L}}_k /{\rm e})^{p^{\ssup{N,L}}_k |\L_N|}} 
\\
&=\e^{o(|\L_N|)} \exp\Big\{-|\L_N| I_L(p^{\ssup{N,L}})\Big\},
\end{aligned}
\end{equation}
where $I_L(p)=\sum_{k=1}^L p_k\log\frac{p_k k}{\Gamma_k\e}$ is the ``$L$-cutoff" version of the rate function $I$ given in \eqref{Idef}.
Hence, we see from the preceding argument that
\begin{equation}\label{firstterm}
\begin{aligned}
&\limsup_{N\to\infty}\frac1{|\L_N|}\log \bigg(\sup_{l\in\mathfrak P_N\colon l/|\L_N|\in B_\delta(p)}\prod_{k=1}^L\frac{\big[|\L_N|\Gamma_{\L_N,k}^{\ssup{\rm bc}}\big]^{l_k}}{l_k! k^{l_k}} \bigg) 
&\leq -\inf_{\widetilde p\in B_\delta(p)}I_L(\widetilde p).
\end{aligned}
\end{equation}

Now the product on $k>L$ on the left-hand side of \eqref{LDPproof}
is decomposed into the product of the $\Gamma$-terms and the remainder. Indeed, we are going to show, for any $\delta\in(0,1)$,
\begin{equation}\label{entropy}
\limsup_{L\to\infty}\limsup_{N\to\infty}\frac 1{|\L_N|}\log \bigg(\sup_{l\in\mathfrak P_N\colon l/|\L_N|\in B_\delta(p)}\prod_{k>L} \Big[\frac 1{l_k!}\Big(\frac 1k|\L_N| \Big)^{l_k}\Big]\bigg)\leq 0.
\end{equation}
Concerning the remainder product of the $\Gamma$-terms,  we will show that 
\begin{equation}\label{Gammaremainder}
\begin{aligned}
&\limsup_{\delta\downarrow0}\limsup_{L\to\infty}\limsup_{N\to\infty}\frac 1{|\L_N|}\log\bigg( \sup_{l\in\mathfrak P_N\colon l/|\L_N|\in B_\delta(p)}\prod_{k>L} \big[\Gamma_{\Lambda_N,k}^{\ssup{\mathrm{bc}}}\big]^{l_k}\bigg)
&\leq-\Big(\rho-\sum_{k\in\N}kp_k\Big)\log\lambda_{\rm c}(\beta).
\end{aligned}
\end{equation}

For proving \eqref{entropy}, we use the Stirling bound 
$$
\Big(\frac{l}\e\Big)^{-l}c\frac 1{\sqrt{l}}\leq \frac 1{l!}\leq \Big(\frac {l }\e\Big)^{-l}C\frac 1{\sqrt{l}} \qquad l\in \N,
$$
for some $c,C\in(0,\infty$. Then we see that the left-hand side of \eqref{entropy} can be bounded as follows.
$$
\begin{aligned}
\frac1{|\L_N|}\log\prod_{k>L} \Big[\frac 1{l_k!}\Big(\frac 1k|\L_N| \Big)^{l_k}\Big]
\leq\frac\rho N\sum_{k>L}l_k\log\frac{N\e}{\rho k l_k}+\frac 1N \rho\log C
\end{aligned}
$$
Now use that $\sum_{k>L}\frac{l_k}N\leq \frac 1L\sum_{k>L}\frac{k l_k}N\leq \frac 1L$ and apply Jensen's inequality to the logarithm and the sum over $l_k/\sum_{n>L}l_n$, to see that \eqref{entropy} holds.

Let us turn to a proof of \eqref{Gammaremainder}. Recall from \eqref{lambdadef} that $\lim_{k\to\infty}\Gamma_k^{-1/k}=\lambda_{\rm c}(\beta)$. Pick an $\eps>0$, then we may assume that $L$ is so large that $\Gamma_k\leq (\lambda_{\rm c}(\beta)^{-1}+\eps)^k$ for any $k>L$. 

Recall that by \eqref{simpleesti} we have $\Gamma_{\L_N,k}^{\ssup{\rm Dir}}\leq \Gamma_k$. Thus we can bound the left-hand side of \eqref{Gammaremainder} as follows, for any $l\in\mathfrak P_N$ and for $N>N_0$.
\begin{equation}\label{secondterm}
\begin{aligned}
\frac 1{|\L_N|}\log \prod_{k>L} \big[\Gamma_{\Lambda_N,k}^{\ssup{\mathrm{bc}}}\big]^{l_k}
&\leq\frac 1{|\L_N|}\sum_{k>L}l_k\log\Big(\big[
(\lambda_{\rm c}(\beta)^{-1}+\eps)\big]^k\Big)\\
&\leq \log \big[
(\lambda_{\rm c}(\beta)^{-1}+\eps)\big]\sum_{k>L}\frac{k l_k}{|\L_N|}\\
&=\log \big[
(\lambda_{\rm c}(\beta)^{-1}+\eps)\big]\Big(\rho-\sum_{k=1}^L k p_k^{\ssup{N,L}}\Big),
\end{aligned}
\end{equation}
with $p_k^{\ssup{N,L}}$ as in \eqref{pNLdef}.

Collecting \eqref{firstterm}, \eqref{entropy} and \eqref{secondterm}, we see that, for any $L\in\N$ and $\delta,\eps\in(0,1)$,
\begin{equation}\label{fastfertig}
\begin{aligned}
\limsup_{N\to\infty}\frac1{|\L_N|}&\log \sup_{l\in\mathfrak P_N\colon l/|\L_N|\in B_\delta(p)}\prod_{k\in \N}\frac{\big[|\L_N|\Gamma_{\L_N,k}^{\ssup{\rm bc}}\big]^{l_k}}{l_k! k^{l_k}}\\
&\leq  -\inf_{\widetilde p\in B_\delta(p)}I_L(\widetilde p)+\Big(\rho-\inf_{\widetilde p\in B_\delta(p)}\sum_{k=1}^L k \widetilde p_k\Big) \log \big[
(\lambda_{\rm c}(\beta)^{-1}+\eps)\big].
\end{aligned}
\end{equation}
Letting $\delta\downarrow 0$ and $\eps\downarrow 0$, we obtain, for any $L\in\N$,
$$
\limsup_{\delta\downarrow 0}\big({\rm l.h.s.~of }\ \eqref{fastfertig}\big)\leq -I_L(p)-\Big(\rho-\sum_{k=1}^L kp_k\Big)\log\lambda_{\rm c}(\beta).
$$
Passing to the limit $L\to\infty$, we complete the proof of \eqref{LDPproof}. Now let us turn to the 

\medskip

\noindent{\bf Proof of \eqref{LDPprooflower}.} Again, we pick a large auxiliary parameter $L$ and define $l^{\ssup{\delta,N}}=(l^{\ssup{\delta,N}}_k)_{k\in\{1,\dots,N\}}$ (also depending on $L$) by
$$
l^{\ssup{\delta,N}}_k=
\begin{cases}
\lfloor |\L_N| p_k\rfloor &\mbox{for }k\leq L,\\
A_{p,L,N}&\mbox{for }k=k_N=\lfloor N^{1/2d}\rfloor,\\
0&\mbox{otherwise,}
\end{cases}
$$
were $A_{p,L,N}$ is picked such that $\sum_{k=1}^N k l^{\ssup{\delta,N}}_k=|\L_N|\rho$. Then $l^{\ssup{\delta,N}}_k/|\L_N|\to p_k$ as $N\to\infty$ for $k\in\{1,\dots,L\}$, and $l^{\ssup{\delta,N}}_{k_N}/|\L_N|=A_{p,L,N}/|\L_N|=
(\rho-\sum_{k=1}^Lk p_k)(1+o(1))
\frac 1{k_N}$ as $N\to\infty$. We see that $l^{\ssup{\delta,N}}/|\L_N|\in B_\delta(p)$ for $N$ and $L$ sufficiently large.

As in \eqref{mainpart}, we see  that
$$
\lim_{N\to\infty}\frac1{|\L_N|}\log \prod_{k=1}^L\frac{\big[|\L_N|\Gamma_{\L_N,k}^{\ssup{\rm bc}}\big]^{l^{\ssup{\delta,N}}_k}}{l^{\ssup{\delta,N}}_k! k^{l^{\ssup{\delta,N}}_k}}
\geq - I_L(p)
$$
Furthermore, it is easy to see that
$$
\liminf_{L\to\infty}\liminf_{N\to\infty}\frac 1{|\L_N|}\log \prod_{k>L} \Big[\frac 1{l^{\ssup{\delta,N}}_k!}\Big(\frac 1k|\L_N| \Big)^{l^{\ssup{\delta,N}}_k}\Big]=0.
$$
Finally, usin the assertion around \eqref{Gammalowbound} (which holds for both bc$=$per and bc$=$Dir) and recalling that $\Gamma_k=\lambda_{\rm c}(\beta)^{-k(1+o(1))}$, we get that
$$
\liminf_{L\to\infty}\liminf_{N\to\infty}\frac 1{|\L_N|}\log \prod_{k>L} \big[\Gamma_{\Lambda_N,k}^{\ssup{\mathrm{bc}}}\big]^{l^{\ssup{\delta,N}}_k}
\geq
-\Big(\rho-\sum_{k\in\N}kp_k\Big)\log\lambda_{\rm c}(\beta).
$$
Putting together the last three displays, we see that \eqref{LDPprooflower} holds.

Hence, \eqref{LDPball} has been proved, and this finishes the proof of the LDP in the assertions (i) and (ii) of Theorem~\ref{thm-limfree}.

\subsection{Analysis of the Rate Function}\label{sec-Varformproof}

\noindent Now we prove assertions (iii) and (iv) of Theorem~\ref{thm-limfree}. It is clear that the state space $\mathcal X_\rho$ is compact and that $J$ is continuous and strictly convex on $\Xcal_\rho$. Hence, it suffices to derive the Euler--Lagrange equations for minimisers. Abbreviate $\alpha_k=k/\Gamma_k $ and $\alpha=(\alpha_k)_k$, then $J(p)=\rho\log\lambda_{\rm c}(\beta)+ \langle p, \log (p\alpha/\e)-{\rm id}\lambda_{\rm c}(\beta)\rangle$ for any $p\in\Xcal_\rho$. Assume that $p$ is a minimizer for $J$. It is easy to see that $p_k>0$ for any $k$, since otherwise adding a small weight would clearly make the $J$-value smaller.  Then, for any perturbator $\delta=(\delta_k)_k\in\R^{\N}$, we have
\begin{equation}\label{ELeqcalc}
0=\partial_\eps\Big|_{\eps=0} J(p+\eps\delta)=\langle\delta,\log (p \alpha/\e)-{\rm id}\lambda_{\rm c}(\beta)\rangle+\langle p,\smfrac\delta p\rangle=\langle\delta,\log (p \alpha)-{\rm id}\lambda_{\rm c}(\beta)\rangle.
\end{equation}
Let us assume that the minimizer $p$ satisfies $\sum_k kp_k=\rho$. Then we may assume that the perturbator $\delta $  satisfies $0=\sum_k k\delta_k=\langle \delta,{\rm id}\rangle$, and the right-hand side of \eqref{ELeqcalc} is equal to $\langle\delta,\log (p \alpha)\rangle$. Hence, we have derived that any $\delta$ that is perpendicular to the vector ${\rm id}= (k)_k$ is also perpendicular to the vector $\log (p\alpha)$. Hence, there is some constant $C$ such that $Ck=\log(p_k\alpha_k)$ for any $k$, which reads
\begin{equation}\label{ELeq}
k p_k=\e^{-Ck}\Gamma_k,\qquad k\in\N.
\end{equation}
Now check the constraint $\sum_k k p_k=\rho$. In the first case, where $\sum_k \Gamma_k \lambda_{\rm c}(\beta)^k$ is finite, the smallest $C$ that can ever make the sum of $\e^{-Ck}\Gamma_k$  finite is $C=-\log\lambda_{\rm c}(\beta)$, and the largest $\rho$ that can ever be realized by that sum is $\rho_{\rm c}(\beta)$. Substituting $\e^{-Ck}=\lambda_{\rm c}(\beta)^{k}\e^{-c(\rho)k}$ we arrive at the Euler-Lagrange equations in \eqref{pkELequation} with $c=c(\rho)$ as in the text below \eqref{pkELequation}. In the case where  $\sum_k \Gamma_k \lambda_{\rm c}(\beta)^k=\infty$, then any $\rho>0$ can be realized by that choice of $C=c(\rho)+\log \lambda_{\rm c}(\beta)$.

Now, if the minimizer $p$ satisfies $\sum_k kp_k<\rho$, then the perturbator $\delta$ is arbitrary, and from \eqref{ELeqcalc} we get that $\log (p \alpha)-{\rm id}\lambda_{\rm c}(\beta)=0$, i.e., that \eqref{ELeq} holds with $\e^{-Ck}$ replaced by $\lambda_{\rm c}(\beta)^{k}$, which proves \eqref{pkELequation} as well. In particular, $\sum_k kp_k=\rho_{\rm c}(\beta)$. The formula in (iv) is proved by substituting what we have derived.

\section{Proof of Theorem~\protect\ref{thm-Gammak}: Estimates on the Green Function}

\label{sec-ProofBEC}

This section is devoted to the proof of Theorem~\ref{thm-Gammak}. Without
loss of generality, we are going to do the proof only for $\beta=1$
(otherwise, consider a simple rescaled version of $v$).

In Section~\ref{sec-outlineThm} we introduce some notation, formulate the
main steps of the proof and prove Theorem~\ref{thm-Gammak}, subject to these
steps. In Section~\ref{sec-OutlineProp} we prove explain the main one of
them, and prove it finally in Section~\ref{sec-def-Pi}.

\subsection{Some Notation, and the Main Steps}

\label{sec-outlineThm}

\noindent Recall that $\varphi _{t}(x)=(2\pi t)^{-d/2}\exp [-\frac{|x|^{2}}{%
2t}]$ denotes the Gaussian kernel with variance $t$, and we write $\varphi
:=\varphi _{1}$. We define%
\begin{equation}
G(x):=\sum_{n=1}^{\infty }\varphi _{n}(x),\qquad x\in \mathbb{R}^{d},
\label{Def_G}
\end{equation}%
which is, up to the missing summand for $n=0$ (which is often formally
written as the delta function at $x$), equal to the Green function of the
random walk with standard normal increments. (It is not to be confused with
the Green function $G_{\mathrm{c}}(x)=\int_{0}^{\infty }\varphi _{t}(x)%
\,\mathrm{d}t$ for Brownian motion. See Lemma~\ref{Le_Edgeworth} and its proof
for some relation between them.) In spite of the missing summand for $n=0$,
we will call $G$ the Green function from now. Note that $G(x)$ is finite in $%
d\geq 3$, and $G(\cdot )$ is rotational symmetric, bounded, and infinitely
often differentiable with all derivatives in the set ${\mathcal{C}}_{0}(%
\mathbb{R}^{d})$ of continuous functions $\mathbb{R}^{d}\rightarrow \mathbb{R%
}$ vanishing at infinity. Recall that $v$ is a bounded continuous function $%
\mathbb{R}^{+}\rightarrow \mathbb{R}^{+}$ with compact support, i.e., we
assume that there exists $L,R>0$ such that $v(r)\leq L$ for every $r,$ and $%
v(r)=0$ for $r\geq R.$

Let us introduce the main object of the proof, a version of the
self-repellent Brownian motion. Recall the Hamiltonian on ${\mathcal{C}}_{N}$%
:%
\[
H_{N}(B):=\sum_{1\leq i<j\leq N}V(B_{i},B_{j}),\ \mathrm{where\ }%
V(f,g)=\int_{0}^{1}v(|f(s)-g(s)|)\,\mathrm{d}s,\ \,\,f,g\in \mathcal{C}_{1},
\]%
where we recall from \eqref{jthleg} that $B_{j}=(B_{j}(s))_{s\in \lbrack
0,1]}:=(B((j-1)+s))_{s\in \lbrack 0,1]}$ is the $j$-th leg of the Brownian
motion $B$. Then $0\leq H_{N}(B)\leq LN$ since $0\leq v\leq L$. Also recall
that $\mu _{x,y}^{{\scriptsize (1)}}$ denotes the canonical Brownian bridge
measure in the time interval $[0,1]$. We write $\mathbb{P}_{N}$ for the
standard Wiener measure on $\mathcal{C}_{N}$. Then the tilted measure ${%
\mathbb{Q}}_{\alpha ,N}$ on $\mathcal{C}_{N}$ is defined by%
\begin{equation}
{\mathbb{Q}}_{\alpha ,N}(\mathrm{d}B):=\mathrm{e}^{-\alpha H_{N}(B)}\,%
\mathbb{P}_{N}(\mathrm{d}B),\qquad N\in \mathbb{N},\alpha \in \lbrack
0,\infty ).  \label{def-tilted}
\end{equation}%
Note that ${\mathbb{Q}}_{\alpha ,N}$ is not normalized. Its normalized
version is the announced variant of the self-repellent Brownian motion. For $%
N=1,$ there is no interaction, so we have ${\mathbb{Q}}_{\alpha ,1}=\mathbb{P%
}_{1}$. The finite dimensional marginals of ${\mathbb{Q}}_{\alpha ,N}$ have
continuous Lebesgue densities. In particular, we define%
\begin{equation}
\Gamma _{\alpha ,N}(x):=\frac{{\mathbb{Q}}_{\alpha ,N}(B_{N}\in \mathrm{d}x)%
}{\mathrm{d}x}.  \label{def-GammaN}
\end{equation}%
Then $\Gamma _{k}$ defined in \eqref{Gammadef} for $v$ replaced by $\alpha v$
(and $\beta =1$) is equal to $\Gamma _{\alpha ,k}(0)$. Note that $\mathrm{e}%
^{-\alpha LN}\varphi _{N}(x)\leq \Gamma _{\alpha ,N}(x)\leq \varphi _{N}(x)$
for any $N\in \mathbb{N}$ and $x\in \mathbb{R}^{d}$. Now define the Green
function for the self-repellent Brownian motion as follows: for $\lambda >0,$
let%
\begin{equation}
G_{\alpha ,\lambda }^{{\scriptsize (\Gamma )}}(x):=\sum_{N=1}^{\infty
}\lambda ^{N}\Gamma _{\alpha ,N}(x),\ \lambda _{\mathrm{cr}}(\alpha ):=\sup
\left\{ \lambda >0\colon \Vert G_{\alpha ,\lambda }^{{\scriptsize (\Gamma )}%
}\Vert _{1}<\infty \right\} .  \label{Def_criticalvalue}
\end{equation}%
Then $\lambda _{\mathrm{cr}}(\alpha )$ is equal to $\lambda _{\mathrm{c}}(1)$
defined in \eqref{lambdadef} for $v$ replaced by $\alpha v$.

Our main result in the present section is:

\begin{theorem}
\label{Th_main} Assume $d\geq 5$. Then there exists $\alpha _{0}=\alpha
_{0}(v,d)>0$ such that for $\alpha \in (0,\alpha _{0})$%
\[
G_{\alpha ,\lambda _{\mathrm{cr}}(\alpha )}^{{\scriptsize (\Gamma )}}(x)\leq
5G(x),\qquad x\in \mathbb{R}^{d}. 
\]%
In particular, $\sum_{N=1}^{\infty }\lambda _{\mathrm{cr}}(\alpha
)^{N}\Gamma _{\alpha ,N}(0)<\infty $ for $\alpha \in (0,\alpha _{0})$.
\end{theorem}

Theorem \ref{Th_main} evidently implies Theorem \ref{thm-Gammak}. The rest
of the article is devoted to the proof of Theorem \ref{Th_main}.

For $\lambda <\lambda _{\mathrm{cr}}(\alpha )$, consider%
\begin{equation}
\overline{G}_{\alpha ,\lambda }^{{\scriptsize (\Gamma )}}:=G_{\alpha
,\lambda }^{{\scriptsize (\Gamma )}}\star \varphi ,  \label{Def_Gbar}
\end{equation}%
so that $\overline{G}_{\alpha ,\lambda }^{{{({\Gamma }})}}$ is integrable,
and has bounded and continuous derivatives of all orders. We first observe
the following simple fact:

\begin{lemma}
\label{Le_0}Fix $\alpha \in (0,\infty )$ such that $\mathrm{e}^{\alpha
L}\leq 2$ (Here, $2$ is taken for convenience and any fixed constant $>1$
would suffice for our purposes). Then the function 
\[
f_{\alpha }(\lambda ):=\sup_{x\in \mathbb{R}^{d}}\frac{\overline{G}_{\alpha
,\lambda }^{{\scriptsize (\Gamma )}}(x)}{G(x)}
\]%
is continuous and increasing in $\lambda \in \lbrack 0,\lambda _{\mathrm{cr}%
}(\alpha )).$
\end{lemma}

\begin{proof}
It is clear that $f_\alpha$ is increasing. For the continuity, it suffices
to prove that $f_\alpha(\cdot)$ is continuous on $[ \lambda _{0}^{\prime
},\lambda _{0} ] $ for any $0<\lambda _{0}^{\prime }<\lambda _{0}<\lambda _{%
\mathrm{cr}}( \alpha).$

We first prove that there exists $R(\lambda _{0}^{\prime },\lambda _{0})$
such that for $\lambda \in \lbrack \lambda _{0}^{\prime },\lambda _{0}]$%
\begin{equation}
f_{\alpha }(\lambda )=\sup_{|x|\leq R(\lambda _{0}^{\prime },\lambda _{0})}%
\frac{\overline{G}_{\alpha ,\lambda }^{{\scriptsize (\Gamma )}}(x)}{G(x)}.
\label{Cont_campact}
\end{equation}%
Since $\mathrm{e}^{-\alpha LN}\varphi _{N}\leq \Gamma _{\alpha ,N}$, and
since $\sum_{n}\lambda ^{n}\varphi _{n}$ does not converge for $\lambda >1,$
we get that $\lambda _{\mathrm{cr}}(\alpha )\leq \mathrm{e}^{\alpha L}\leq 2$%
, by the choice of $\alpha $. Note that the lower bound $\lambda _{\mathrm{cr%
}}(\alpha )\geq 1$ is evident as $H_{N}(B)\geq 0.$ Define $\lambda
_{1}:=(\lambda _{0}+\lambda _{\mathrm{cr}}(\alpha ))/2.$ For any $N\in 
\mathbb{N}$, using that $\lambda _{\mathrm{cr}}(\alpha )\leq 2$ and $\Gamma
_{\alpha ,n}(\cdot )\leq \varphi _{n}(\cdot )$, 
\begin{eqnarray}
\overline{G}_{\alpha ,\lambda }^{{{({\Gamma }})}}(x)
&=&\sum_{n=1}^{N-1}\lambda ^{n}(\Gamma _{\alpha ,n}\star \varphi
)(x)+\sum_{n=N}^{\infty }\left( \frac{\lambda }{\lambda _{1}}\right)
^{n}\lambda _{1}^{n}(\Gamma _{\alpha ,n}\star \varphi )(x)  \nonumber \\
&\leq &\sum_{n=1}^{N-1}2^{n}\varphi _{n+1}(x)+\sum_{n=N}^{\infty }\left( 
\frac{\lambda _{0}}{\lambda _{1}}\right) ^{n}\lambda _{1}^{n}(\Gamma
_{\alpha ,n}\star \varphi )(x)  \label{G_lambda_bar} \\
&\leq &\sum_{n=1}^{N-1}2^{n}\frac{1}{(2\pi (n+1))^{\frac{d}{2}}}\,\exp
[-|x|^{2}/2(n+1)]+\left( \frac{\lambda _{0}}{\lambda _{1}}\right) ^{N}\Vert 
\overline{G}_{\alpha ,\lambda _{1}}^{{{({\Gamma }})}}\Vert _{\infty }. 
\nonumber
\end{eqnarray}%
Remark that by definition of $\lambda _{1}$ we have $G_{\alpha ,\lambda
_{1}}^{{{({\Gamma }})}}\in L^{1}(\mathbb{R}^{d})$ and therefore $\Vert 
\overline{G}_{\alpha ,\lambda _{1}}^{{{({\Gamma }})}}\Vert _{\infty }=\Vert
G_{\alpha ,\lambda _{1}}^{{{({\Gamma }})}}\star \varphi \Vert _{\infty
}<\infty $. By choosing%
\[
N:=\frac{|x|}{2\sqrt{2\log 2}} 
\]%
one sees that there are constants $c,C$, depending on $\lambda _{0},$ such
that for $\lambda \leq \lambda _{0}$%
\[
\overline{G}_{\alpha ,\lambda }^{{{({\Gamma }})}}(x)\leq C\exp [-c|x|]. 
\]%
As $\overline{G}_{\alpha ,\lambda }^{{{({\Gamma }})}}(0)$ is bounded away
from $0$ for $\lambda \geq \lambda _{0}^{\prime }$, and the Green function $%
G(x)$ decays like $\mathrm{const}\times |x|^{-d+2}$ as $|x|\rightarrow
\infty $, \eqref{Cont_campact} follows.

To prove the desired continuity, we appeal to the representation of $%
\overline{G}_{\alpha ,\lambda }^{{{({\Gamma }})}}$ in the first identity in (%
\ref{G_lambda_bar}). Then $(\Gamma _{\alpha ,n}\star \varphi )(x)$ is
continuous, and therefore uniformly continuous on $B_{R}:=\{x\colon |x|\leq
R\}.$ Therefore, for any $N,$ the first part $\sum_{n=1}^{N-1}\lambda
^{n}(\Gamma _{\alpha ,n}\star \varphi )(\cdot )$ is uniformly continuous,
also uniformly in $\lambda \in \lbrack \lambda _{0}^{\prime },\lambda _{0}].$
Now since $\lambda _{0}<\lambda _{\mathrm{cr}}$, by definition of $\lambda _{%
\mathrm{cr}}(\alpha )$, the second summand $\sum_{n=N}^{\infty }\lambda
_{0}^{n}(\Gamma _{\alpha ,n}\star \varphi )(x)$ converges to $0$ as $%
N\rightarrow \infty $ uniformly in $x$ and $\lambda \in \lbrack \lambda
_{0}^{\prime },\lambda _{0}],$ we conclude that $\overline{G}_{\alpha
,\lambda }^{{{({\Gamma }})}}(x)$ is uniformly continuous in $x\in B_{R}$ and 
$\lambda \in \lbrack \lambda _{0}^{\prime },\lambda _{0}].$ The same remains
true for $\overline{G}_{\alpha ,\lambda }^{{{({\Gamma }})}}(x)/G(x)$.
Therefore $\sup_{x\in B_{R}}\overline{G}_{\alpha ,\lambda }^{{{({\Gamma }})}%
}(x)/G(x)$ is continuous in $\lambda \in \lbrack \lambda _{0}^{\prime
},\lambda _{0}]$, as claimed.
\end{proof}

The key to proof of Theorem \ref{Th_main} is the following.

\begin{prop}
\label{Prop_main}There exists $\alpha _{0}>0$ such that for $\alpha \leq
\alpha _{0},$ there is no $\lambda <\lambda _{\mathrm{cr}}(\alpha )$ with $%
f_{\alpha }(\lambda )\in (2,3].$
\end{prop}

The proof is based on two main lemmas, namely Lemma \ref{Le_1} and Lemma \ref%
{Le_2}, which will be proved in the next two sections. Let us first note
that the above proposition, together with Lemma~\ref{Le_0}, implies Theorem %
\ref{Th_main}:

\begin{proof}[\textbf{Proof of Theorem \protect\ref{Th_main} (assuming
Proposition \protect\ref{Prop_main})}]
We will show Theorem \ref{Th_main} with $\alpha_0$ replaced by $%
\min\{\alpha_0,\frac 1L\log 2\}$, where $\alpha_0$ is from Proposition~\ref%
{Prop_main}.

Fix $\alpha \in (0,\alpha _{0})$ so small that $\lambda _{\mathrm{cr}%
}(\alpha )\leq 2$. As $\Gamma _{\alpha ,n}\leq \varphi _{n}$ for any $n\in 
\mathbb{N}$, it is evident that $f_{\alpha }(\lambda )\leq 1$ for $\lambda
\leq 1$. Then Proposition \ref{Prop_main} and the continuity in $\lambda $
shown in Lemma \ref{Le_0} imply that $f_{\alpha }(\lambda )\leq 2$ for $%
\lambda <\lambda _{\mathrm{cr}}(\alpha ).$ This means that $\overline{G}%
_{\alpha ,\lambda }^{{{({\Gamma }})}}(x)\leq 2G(x)$ for any $x\in \mathbb{R}%
^{d}.$ As $\Gamma _{\alpha ,n}(x)\geq 0$ for all $x,$ it follows that $%
\overline{G}_{\alpha ,\lambda _{\mathrm{cr}}(\alpha )}^{{{({\Gamma }})}%
}(x)\leq 2G(x)$. Observe now that from $H_{n}(B)\leq H_{n+1}(B)$, it follows
that%
\[
\Gamma _{\alpha ,n}\star \varphi \geq \Gamma _{\alpha ,n+1}. 
\]%
This implies%
\[
\overline{G}_{\alpha ,\lambda _{\mathrm{cr}}(\alpha )}^{{{({\Gamma }})}}\geq 
\frac{G_{\alpha ,\lambda _{\mathrm{cr}}}^{{{({\Gamma }})}}-\varphi }{\lambda
_{\mathrm{cr}}(\alpha )}, 
\]%
and therefore, using $\overline{G}_{\alpha ,\lambda _{\mathrm{cr}}(\alpha
)}^{{{({\Gamma }})}}(x)\leq 2G(x),$ and $\lambda _{\mathrm{cr}}(\alpha )\leq
2,$ it follows that%
\begin{equation}
G_{\alpha ,\lambda _{\mathrm{cr}}(\alpha )}^{{{({\Gamma }})}}\leq \lambda _{%
\mathrm{cr}}(\alpha )\overline{G}_{\alpha ,\lambda _{\mathrm{cr}}(\alpha )}^{%
{{({\Gamma }})}}+\varphi \leq 4G+\varphi \leq 5G,  \label{5_G_est}
\end{equation}%
which implies Theorem \ref{Th_main}.
\end{proof}

\subsection{Outline of the Proof of Proposition \protect\ref{Prop_main}.}\label{sec-OutlineProp}

\noindent We outline the arguments for Proposition \ref{Prop_main}. We have
to show that if $\alpha $ is small enough, then the following implication is
true for all $\lambda <\lambda _{\mathrm{cr}}(\alpha )$: 
\[
f_{\alpha }(\lambda )\leq 3\quad \Longrightarrow \quad f_{\alpha }(\lambda
)\leq 2.
\]%
So, we assume $f_{\alpha }(\lambda )\leq 3$. For the ease of notation, we
drop the parameter $\alpha $. The proof then splits into three main steps:
First, by the argument leading to (\ref{5_G_est}) we see that this leads to%
\begin{equation}
G_{\lambda }^{{\scriptsize (\Gamma )}}\leq \lambda \overline{G}_{\lambda }^{%
{\scriptsize (\Gamma )}}+\varphi \leq 7G(x).  \label{Bound_7}
\end{equation}%
Next, we build on a modification of the lace expansion technique. This
provides a complicated but rather explicit bound for an inverse of $%
I+G_{\lambda }^{{{({\Gamma }})}}$, under convolution, where $I$ is the
identity operator (formally given by the Dirac $\delta $ function). This
inverse is of the form 
\[
(I+G_{\lambda }^{{\scriptsize (\Gamma )}})^{-1}=I-G_{\lambda }^{{\scriptsize %
(\Pi )}},\ \mathrm{where\ }G_{\lambda }^{{\scriptsize (\Pi )}%
}(x):=\sum_{N=1}^{\infty }\lambda ^{N}\Pi _{N}(x)
\]%
for some functions $\Pi _{N}(x)$ that we will describe precisely in the next
section. The crucial point is that under the condition $G_{\lambda }^{{{({%
\Gamma }})}}(x)\leq 7G(x)$ for all $x\in \mathbb{R}^{d}$, we obtain good
decay properties for $G_{\lambda }^{{\scriptsize (\Pi )}}(x).$ Also $%
G_{\lambda }^{{\scriptsize (\Pi )}}(x)$ is in a suitable sense small if $%
\alpha $ is small.

In the final step, which is quite orthogonal to lace expansion used in the
preceding step, we show that if $G_{\lambda }^{{\scriptsize (\Pi )}}$ is
small, in a sense that will be made precise, we can invert $I-G_{\lambda }^{%
{\scriptsize (\Pi )}}$ under convolution using a Neumann series. This then
gives an estimate of $G_{\lambda }^{{{({\Gamma }})}}$ in terms of
\textquotedblleft smallness properties\textquotedblright\ of $G_{\lambda }^{{%
{({\Pi }})}},$ and in particular it gives the estimate $G_{\lambda }^{%
{\scriptsize (\Gamma )}}(x)\leq 2G(x)$ if $\alpha $ is small enough, implying%
\[
\overline{G}_{\lambda }^{{\scriptsize (\Gamma )}}=G_{\lambda }^{{\scriptsize %
(\Gamma )}}\star \varphi \leq 2G(x)\star \varphi \leq 2G.
\]%
This completes the proof of Proposition \ref{Prop_main}.

The argument is however slightly more delicate than just indicated, as we
cannot use a Neumann series directly on the $G_{\lambda }^{{\scriptsize (\Pi
)}}.$ There is a correction needed corresponding to the correct diffusion
matrix of the self-repelling process. This relation is however somewhat
hidden in the computation.

\subsection{The Main Lemmas, and Proof of Proposition \protect\ref{Prop_main}%
.}\label{sec-def-Pi}

\noindent For $1\leq k<m,\ k,m\in \mathbb{N}$, let $\mathcal{G}_{k,m}$ be
the set of simple graphs on the set $\{k,k+1,\ldots ,m\}.$ We also write $%
\mathcal{G}_{m}:=\mathcal{G}_{1,m}$ and set%
\begin{equation}
U_{\alpha }(f,g):=1-\mathrm{e}^{-\alpha V(f,g)},\ f,g\in {\mathcal{C}}%
_{1},\qquad \mathrm{and}\,\,\,U_{i,j}^{{{({\alpha }})}}:=U_{\alpha
}(B_{i},B_{j})\qquad i,j\in \mathbb{N}.  \label{Uij}
\end{equation}%
Remark that $\alpha >0$ small implies that $U_{\alpha }$ is small. With $%
H_{N}(B)=\sum_{1\leq i<j\leq N}V(B_{i},B_{j})$ we can expand 
\[
\mathrm{e}^{-\alpha H_{N}(B)}=\prod_{1\leq i<j\leq N}(1-U_{i,j}^{{{({\alpha }%
})}})=\sum_{{\mathfrak{g}}\in \mathcal{G}_{N}}\prod_{(i,j)\in {\mathfrak{g}}%
}(-U_{i,j}^{{{({\alpha }})}}),
\]%
where we identify the graph ${\mathfrak{g}}$ with its vertex set. For
notational convenience, from now on we will suppress the dependence on $%
\alpha $ on the function $\Gamma _{N}=\Gamma _{\alpha ,N}$ defined in %
\eqref{def-GammaN}, on the tilted measure ${\mathbb{Q}}_{N}={\mathbb{Q}}%
_{\alpha ,N}$ defined in \eqref{def-tilted} and on the functions $%
U=U_{\alpha }$ and $U_{i,j}=U_{i,j}^{{({\alpha }})}$ defined in \eqref{Uij}.
Then ${\mathbb{Q}}_{N}$ can be rewritten as 
\begin{equation}
{\mathbb{Q}}_{N}(\mathrm{d}B)=\sum_{{\mathfrak{g}}\in \mathcal{G}%
_{N}}\prod_{(i,j)\in {\mathfrak{g}}}(-U_{i,j})\,\mathbb{P}_{N}(\mathrm{d}B).
\label{Q_N_expansion}
\end{equation}

We call an index $k\in \{1,\ldots ,N\}$ a \emph{breakpoint} of a graph ${%
\mathfrak{g}}\in \mathcal{G}_{N}$ if there is no edge $(i,j)\in {\mathfrak{g}%
}$ with $i\leq k<j$. If a graph has no breakpoint, we call it \emph{%
irreducible}. Remark that we use the notion in a way that is slightly
different from the standard one in lace expansions: For instance, if ${%
\mathfrak{g}}$ contains a bond $(k,j)$ with $j>k,$ then $k$ is not a
breakpoint. We write $\mathcal{J}_{N}$ for the set of irreducible graphs in $%
\mathcal{G}_{N}$. This makes sense only for $N\geq 2.$ For ${\mathfrak{g}}%
\in \mathcal{G}_{N}\backslash \mathcal{J}_{N},$ we denote by $b({\mathfrak{g}%
})$ the smallest breakpoint of ${\mathfrak{g}}$, and write $\mathcal{G}_{N}^{%
{{({k}})}}$ for the set of graphs in $\mathcal{G}_{N}\backslash \mathcal{J}%
_{N}$ for which $k$ is the smallest breakpoint. For convenience, we set $b({%
\mathfrak{g}}):=N$ for ${\mathfrak{g}}\in \mathcal{J}_{N},$ and accordingly $%
\mathcal{G}_{N}^{{{{(N})}}}:=\mathcal{J}_{N}.$

If $\mathbb{E}_{N}$ denotes expectation with respect to the Wiener measure $%
\mathbb{P}_{N}$ on ${\mathcal{C}}_{N}$, we define%
\begin{equation}
\Pi _{N}^{{{({\alpha }})}}(x)=\Pi _{N}(x):=\sum_{{\mathfrak{g}}\in \mathcal{J%
}_{N}}\mathbb{E}_{N}\Big[\prod_{(i,j)\in {\mathfrak{g}}}(-U_{i,j})\mathbf{1}%
_{\{B_{N}\in \mathrm{d}x\}}\Big]\Big/\mathrm{d}x,\qquad N\geq 1,
\label{def-Pi}
\end{equation}%
Remark that for $N=1,\ \mathcal{J}_{1}=\emptyset ,$ and $\Pi _{1}^{{{({%
\alpha }})}}$ simply equals $\varphi .$ Define the Green function of $\Pi $
by%
\begin{equation}
G_{\alpha ,\lambda }^{{\scriptsize (\Pi )}}(x):=\sum_{N=1}^{\infty }\lambda
^{N}\Pi _{N}(x),  \label{G_Pi_Def}
\end{equation}%
provided the series is absolutely summable. The precise formulation for the
second step sketched above is provided by the following lemma.

\begin{lemma}
\label{Le_1}There is a constant $C>0$ (depending only on $d,L,R)$ such that
for any $\alpha ,\lambda >0$ that satisfy%
\begin{equation}
G_{\alpha ,\lambda }^{{\scriptsize (\Gamma )}}(x)\leq 7G(x),\qquad x\in 
\mathbb{R}^{d},  \label{3_bound}
\end{equation}%
(which in particular implies $\lambda \leq \lambda _{\mathrm{cr}}\left(
\alpha \right) $), the following inequality holds: The right-hand side of (%
\ref{G_Pi_Def}) is absolutely summable and 
\begin{equation}
\sum_{N=2}^{\infty }\lambda ^{N}|\Pi _{N}(x)|\leq C\alpha
(1+|x|)^{6-3d},\qquad x\in \mathbb{R}^{d}.  \label{Bound_Laplace_Pi}
\end{equation}
\end{lemma}

\begin{remark}
A crucial point here is that we sum only from $N=2.$ In fact $\Pi
_{1}=\varphi .$
\end{remark}

The proof of Lemma \ref{Le_1} is in Section \ref{Le_1_proof}.

The reason to consider $G_{\lambda }^{{\scriptsize (\Pi )}}$ is that $\delta
_{0}-G_{\lambda }^{{\scriptsize (\Pi )}}$ is essentially an inverse of $%
\delta _{0}+G_{\lambda }^{{\scriptsize (\Gamma )}}$ under convolution. This
is based on the following simple result:

\begin{lemma}
\label{Le_Conv_equation}With $\Gamma _{0}:=\delta _{0}$ and $\Gamma
_{1}=\varphi $ we have 
\begin{equation}
\Gamma _{N}=\sum_{k=1}^{N}\Pi _{k}\star \Gamma _{N-k},\qquad N\geq 1.
\label{Convol_equ}
\end{equation}%
(With the convention $\Gamma _{0}=\delta _{0}$).
\end{lemma}

\begin{proof}
The formula is correct for $N=1$ where $\Gamma _{1}=\Pi _{1}=\varphi .$
Using (\ref{Q_N_expansion}), we have for $N\geq 2$%
\begin{eqnarray*}
\Gamma _{N}(x) &=&\frac{\mathrm{d}}{\mathrm{d}x}\mathbb{E}_{N}\Big[\sum_{{%
\mathfrak{g}}\in \mathcal{G}_{N}}\prod_{(i,j)\in {\mathfrak{g}}}(-U_{ij}){%
\mathbf{1}}_{\{B_{N}\in \mathrm{d}x\}}\Big] \\
&=&\frac{\mathrm{d}}{\mathrm{d}x}\sum_{k=1}^{N}\mathbb{E}_{N}\Big[\sum_{{%
\mathfrak{g}}\in \mathcal{G}_{N}^{(k)}}\prod_{(i,j)\in {\mathfrak{g}}%
}(-U_{ij}){\mathbf{1}}_{\{B_{N}\in \mathrm{d}x\}}\Big].
\end{eqnarray*}

The summand for $k=N$ on the right hand side restricts the summation over ${%
\mathfrak{g}}$ to ${\mathfrak{g}}\in \mathcal{J}_{N}$, which gives for this
summand $\Pi _{N}$. As $\Gamma _{0}=\delta _{0},$ this is the $k=N$ summand
on the right hand side of (\ref{Convol_equ}).

We next consider the summand for $k=1$. We remark that $1$ is a break point
of ${\mathfrak{g}}$ if and only if there is no edge in ${\mathfrak{g}}$
starting in $1$. In the product $\prod_{(i,j)\in {\mathfrak{g}}}(-U_{i,j})$
there is then no interaction between the first leg $B_{1}$ and any other
leg. Using translation invariance, we get%
\[
\frac{\mathrm{d}}{\mathrm{d}x}\sum_{{\mathfrak{g}}\in \mathcal{G}_{N}^{{{({1}%
})}}}\mathbb{E}_{N}\Big[\prod_{(i,j)\in {\mathfrak{g}}}(-U_{i,j}){\mathbf{1}}%
_{\{B_{N}\in \mathrm{d}x\}}\Big]=\int \varphi (y)\Gamma _{N-1}(x-y)\,\mathrm{%
d}y 
\]%
which as $\Pi _{1}=\varphi $ is the first summand on the rhs of (\ref%
{Convol_equ}).

This argument can be extended to handle the summation over graphs in $%
\mathcal{G}_{N}^{{{({k}})}}$, for $2\leq k\leq N-1$. A graph ${\mathfrak{g}}$
is in $\mathcal{G}_{N}^{{{({k}})}}$ if and only if it is the union of an
irreducible graph on $\{1,\ldots ,k\}$ and an arbitrary graph on $%
\{k+1,\ldots ,N\}$. Summing over all the possibilities, we get for $2\leq
k\leq N-1$%
\[
\frac{\mathrm{d}}{\mathrm{d}x}\sum_{{\mathfrak{g}}\in \mathcal{G}_{N}^{{{({k}%
})}}}\mathbb{E}_{N}\Big[\prod\nolimits_{(i,j)\in {\mathfrak{g}}}(-U_{i,j}){%
\mathbf{1}}_{\{B_{N}\in \mathrm{d}x\}}\Big]=\int \Pi _{k}(y)\Gamma
_{N-k}(x-y)\,\mathrm{d}y. 
\]%
We remark that for $k=N-1,$ there is no graph on $\{N\},$ and we get $\Pi
_{N-1}\star \varphi .$ Summing over $k$ we conclude the proof of the lemma.
\end{proof}

\begin{remark}
A consequence of the convolution equation \eqref{Convol_equ} is that%
\begin{eqnarray*}
G_{\lambda }^{{\scriptsize (\Gamma )}} &=&\sum_{N=1}^{\infty }\lambda
^{N}\Gamma _{N}=\sum_{N=1}^{\infty }\lambda ^{N}\sum_{k=1}^{N}\Pi _{k}\star
\Gamma _{N-k} \\
&=&\sum_{k=1}^{\infty }\lambda ^{k}\Pi _{k}\star \bigg(\sum_{N=k}^{\infty
}\lambda ^{N-k}\Gamma _{N-k}\bigg)\newline
=G_{\lambda }^{{\scriptsize (\Pi )}}\star (\delta _{0}+G_{\lambda }^{%
{\scriptsize (\Gamma )}}) \\
&=&G_{\lambda }^{{\scriptsize (\Pi )}}+G_{\lambda }^{{\scriptsize (\Pi )}%
}\star G_{\lambda }^{{\scriptsize (\Gamma )}},
\end{eqnarray*}%
i.e.%
\begin{equation}
-G_{\lambda }^{{\scriptsize (\Pi )}}\star G_{\lambda }^{{\scriptsize (\Gamma
)}}+G_{\lambda }^{{\scriptsize (\Gamma )}}-G_{\lambda }^{{\scriptsize (\Pi )}%
}=0,  \label{Basic_Green_identity}
\end{equation}%
at least if $G_{\lambda }^{{{({\Pi }})}}$ and $G_{\lambda }^{{{({\Gamma }})}}
$ are integrable which is the case for $\lambda <\lambda _{\mathrm{cr}}.$

This equation can be rewritten as%
\[
(\delta _{0}-G_{\lambda }^{{\scriptsize (\Pi )}})\star (\delta
_{0}+G_{\lambda }^{{\scriptsize (\Gamma )}})=\delta _{0}.
\]
\end{remark}

The precise formulation of the third step for the proof of Proposition \ref%
{Prop_main} is provided by the following lemma.

\begin{lemma}
\label{Le_2}Assume that $\left\{ \Pi _{N}\right\} _{N\geq 2}$ satisfies \eqref%
{Bound_Laplace_Pi} for $\lambda <\lambda _{\mathrm{cr}}(\alpha )$ with the
constant $C$ depending only on $d,L,R.$ If $\alpha $ is small enough, then
for every $\lambda <\lambda _{\mathrm{cr}}(\alpha )$ there exists a unique
continuous, bounded, rotational symmetric function $F_{\alpha }\in L_{1}$
that satisfies%
\begin{equation}
-G_{\alpha ,\lambda }^{{\scriptsize (\Pi )}}\star F_{\alpha ,\lambda
}+F_{\alpha ,\lambda }=G_{\alpha ,\lambda }^{{\scriptsize (\Pi )}}
\label{Eq_F}
\end{equation}%
and 
\begin{equation}
F_{\alpha ,\lambda }(x)\leq 2G(x),\qquad x\in \mathbb{R}^{d}.
\label{bound_F}
\end{equation}
\end{lemma}

The proof of Lemma \ref{Le_2} is given in Section \ref{Le_2_proof}.

\begin{proof}[\textbf{Proof of Proposition \protect\ref{Prop_main} (assuming
Lemma \protect\ref{Le_1} and Lemma \protect\ref{Le_2})}]
If $f_{\alpha }(\lambda )\leq 3,$ the argument leading to (\ref{5_G_est})
implies that $G_{\alpha ,\lambda }^{{{({\Gamma }})}}\leq 7G$. By Lemma \ref%
{Le_1}%
\[
\sup_{x}\frac{\left\vert \sum\nolimits_{N=2}^{\infty }\lambda ^{N}\Pi
_{N}\left( x\right) \right\vert }{(1+|x|)^{6-3d}}\leq C\alpha .
\]%
By Lemma \ref{Le_2}, if $\alpha $ is small enough, there is, for any $%
\lambda <\lambda _{\mathrm{cr}}(\alpha )$, a unique $F_{\alpha ,\lambda }$
that satisfies \eqref{Eq_F}. By (\ref{Basic_Green_identity}), we have $%
F_{\alpha ,\lambda }=G_{\alpha ,\lambda }^{{{({\Gamma }})}},$ and then by %
\eqref{bound_F}, $G_{\lambda }^{{{({\Gamma }})}}\leq 2G.$ This implies that $%
f_{\alpha }(\lambda )\leq 2$ and finishes the proof.
\end{proof}

\section{Proof of Lemma \protect\ref{Le_1}: Lace Expansion\label{Le_1_proof}}

\noindent In the standard version of lace expansions, one considers sets of
paths that visit a single point on the lattice possibly several times. This
is not the case in our situation, since the steps of our random walk have a
Lebesgue density. What we do instead is to consider continuous-time paths
that on time intervals of order one can be close together. The closeness is
measured by the following functions $u_{n,\alpha },\ n\geq 2,$ depending on $%
2n$ locations:%
\begin{equation}
u_{n,\alpha }(\mathbf{x}):=\int_{\mathcal{C}_{1}^{n}}\prod_{i=1}^{n-1}U_{%
\alpha }(f_{i},f_{i+1})\prod_{i=1}^{n}\mu _{x_{2i-1},x_{2i}}^{\ssup 1}(%
\mathrm{d}f_{i})\ \mathrm{for}\,\,\mathbf{x}\in (\mathbb{R}^{d})^{2n},
\label{def-un}
\end{equation}%
where $\mu _{x,y}^{\ssup 1}$ is defined in  \eqref{nnBBM}.

In words, this is (up to normalization) the expectation over a vector of $n$
independent Brownian bridges from $x_{2i-1}$ to $x_{2i}$ for $i=1,\dots ,n$
with interactions between each two subsequent ones of them. Clearly, $%
u_{n,\alpha }$ is invariant under shifting the components of $\mathbf{x}%
=(x_{1},\ldots ,x_{2n})$ by a single vector $y\in \mathbb{R}^{d}$. Also, $%
U_{\alpha }(f,g)$ is close to $0$ unless $f$ and $g$ are close together.
Therefore, $\ u_{n,\alpha }(\mathbf{x})\approx 0$ unless the components of $%
\mathbf{x}$ are close together. $\ u_{n,\alpha }$ is also small of $\alpha $
is small. The quantitative estimates are expressed in Lemma \ref{Le_Est_u}
in Section \ref{Sect_Technical_est}.

Again for notational simplicity, we will usually drop the parameter $\alpha $
appearing in $u_{n,\alpha }$, but the reader should be aware that the
parameter is present in essentially every expression.\hfill 

The basic reason for using the modification of the usual lace expansion is
that if in the expansion (\ref{Q_N_expansion}) a graph has two bonds which
share a point $k\in \mathbb{N}$, then this part in (\ref{Q_N_expansion})
still contains for the continuous time process a complicated interaction on
the interval $\left[ k-1,k\right] .$ For that reason we require for $k$
being a break point that there is a time gap of length $1.$

\subsection{The Set of Laces $\mathbf{{\mathcal{L}}_{\mathit{N}}}$ and Their
Characterization.}

Recall from Section \ref{sec-def-Pi} that $\mathcal{G}_{N}$ denotes the set
of simple graphs on $\{1,\dots ,N\}$ and $\mathcal{J}_{N}$ denotes the set
of irreducible graphs (i.e., graphs with no breakpoint) in $\mathcal{G}_{N}$
for $N\geq 2$. We also recall the function $\Pi _{N}$ defined in %
\eqref{def-Pi}. For the estimation of $\Pi _{N}$, we are going to introduce
now a decomposition (called a \emph{lace expansion}) of $\mathcal{J}_{N}$
according to the minimal number of edges the graph has so that it is still
irreducible. We write $\mathcal{L}_{N}$ for the set of \textbf{laces}, i.e.,
irreducible graphs that are no longer irreducible when leaving out any edge.
Every non-irreducible graph may have many laces, but we are going to
construct for any ${\mathfrak{g}}\in {\mathcal{J}}_{N}$ a particular lace $%
\mathrm{lace}({\mathfrak{g}})$ that will be crucial in our expansion of $\Pi
_{N}$. Indeed, the starting point of our lace expansion is 
\begin{equation}
\sum_{{\mathfrak{g}}\in {\mathcal{J}}_{N}}\prod_{(i,j)\in {\mathfrak{g}}%
}(-U_{ij})=\sum_{\ell \in {\mathcal{L}}_{N}}\sum_{{\mathfrak{g}}\in {%
\mathcal{J}}_{N}\colon \mathrm{lace}({\mathfrak{g}})=\ell }\prod_{(i,j)\in {%
\mathfrak{g}}}(-U_{ij}),  \label{laceexp1}
\end{equation}%
where ${\mathcal{L}}_{N}$ denotes the set of laces.

Evidently, if the graph on $\{1,\ldots ,N\}$ contains the edge $(1,N)$ then
it is irreducible. We denote the set of graphs that contain this particular
edge by $\mathcal{J}_{N}^{{{({1}})}}$, and for a graph ${\mathfrak{g}}\in 
\mathcal{J}_{N}^{{{({1}})}},\ \mathrm{lace}({\mathfrak{g}}):=\left\{ \left(
1,N\right) \right\} .$

For other irreducible graphs, the construction of its lace is more
complicated. Take ${\mathfrak{g}}\in \mathcal{J}_{N}\setminus \mathcal{J}%
_{N}^{{{({1}})}}.$ ${\mathfrak{g}}$ contains an edge $(1,j)$, as otherwise
it is not irreducible. We take the largest such $j$, which is not $N$, as
otherwise we would have ${\mathfrak{g}}\in \mathcal{J}_{N}^{{{({1}})}}$, and
denote it by $j_{1}.$ This edge $(1,j_{1})$ is the first member of $\mathrm{%
lace}({\mathfrak{g}})$. For the next, we know that there is an edge $%
(i^{\prime },j^{\prime })$ with $i^{\prime }\leq j<j^{\prime },$ as
otherwise ${\mathfrak{g}}$ would not be connected. We take the \textit{%
largest }\textrm{possible }$j^{\prime }$ and afterwards the smallest
possible $i^{\prime }$ to this $j^{\prime }$. We denote this edge by $%
(i_{2},j_{2})\in {\mathfrak{g}}$. If $j_{2}=N$ we are finished with the
construction of $\mathrm{lace}({\mathfrak{g}}):=\left\{ \left(
1,j_{1}\right) ,\left( i_{2},j_{2}=N\right) \right\} .$ Otherwise, we go on
in this way. Finally, we end up with a collection%
\begin{equation}
\ell =\big\{(i_{1}=1,j_{1}),(i_{2},j_{2}),\ldots ,(i_{k},j_{k}=N)\big\}.
\label{typical_lace}
\end{equation}%
This is the uniquely defined $\mathrm{lace}({\mathfrak{g}})$. We write $%
\mathcal{J}_{N}^{{{({k}})}}$ for the set of irreducible graphs whose lace
has $k$ edges.

\begin{lemma}[Characterization of laces]
\label{lem-laceident} Let $\ell \in \mathcal{L}_{N}$ and ${\mathfrak{g}} \in 
\mathcal{J}_{N}.$ Then $\mathrm{lace} ( {\mathfrak{g}} ) =\ell $ if and only
if all the edges $( i,j ) \in {\mathfrak{g}} \backslash \ell $ have the
property that 
\[
\mathrm{lace} ( \ell \cup \{ ( i,j ) \} ) =\ell . 
\]
\end{lemma}

\begin{proof}
If $\mathrm{lace} ( {\mathfrak{g}} ) =\ell $ then evidently all edges $( i,j
) \in {\mathfrak{g}} \backslash \ell $ satisfy $\mathrm{lace} ( \ell \cup \{
( i,j ) \} ) =\ell .$

For the other direction, assume $\ell \in \mathcal{L}_{N}$ satisfies $\ell
\subset {\mathfrak{g}} ,$ and $\mathrm{lace} ( \ell \cup \{ ( i,j ) \} )
=\ell $ for all $( i,j ) \in {\mathfrak{g}} \backslash \ell .$ We want to
prove that $\mathrm{lace} ( {\mathfrak{g}} ) =\ell .$ (we remind the reader
that ${\mathfrak{g}} $ may have many subsets which are laces). Write $\ell $
in the form (\ref{typical_lace}). If $( 1,j_{1} ) \notin \mathrm{lace} ( {%
\mathfrak{g}} ) ,$ then ${\mathfrak{g}} $ would have to contain an edge $(
1,k ) $ with $k>j_{1},$ and then $\mathrm{lace} ( \ell \cup \{ ( 1,k ) \} )
\neq \ell .$ So $( 1,j_{1} ) \in \mathrm{lace} ( {\mathfrak{g}} ) .$ Next,
if $( i_{2},j_{2} ) \notin \mathrm{lace} ( {\mathfrak{g}} ) , $ we would
have an edge $( i^{\prime },j^{\prime } ) $ with either $j^{\prime }>j_{2},$
or ($j^{\prime }=j_{2}$ and $i^{\prime }<i_{2}$). In both cases $\mathrm{lace%
} ( \ell \cup \{ ( i^{\prime },j^{\prime } ) \} ) \neq \ell ,$ and we
conclude $( i_{2},j_{2} ) \in \mathrm{lace} ( {\mathfrak{g}} ) .$ In this
way it follows that all edges in $\ell $ are members of $\mathrm{lace} ( {%
\mathfrak{g}} ) ,$ i.e., $\ell \subset \mathrm{lace} ( {\mathfrak{g}} ) .$
But then, if $\mathrm{lace} ( {\mathfrak{g}} ) $ would contain an edge $%
\notin \ell ,$ it would not be a lace, by the minimality property of laces.
\end{proof}

We write for $\ell \in \mathcal{L}_{N}$%
\[
\mathcal{C}_{N}(\ell ):=\big\{(i,j)\notin \ell \colon \mathrm{lace}(\ell
\cup \{(i,j)\})=\ell \big\},
\]%
and call each bond in $\mathcal{C}_{N}(\ell )$ \textit{compatible }with $%
\ell $. In this terminology, Lemma \ref{lem-laceident} can therefore be
rephrased by saying 
\begin{equation}
\mathrm{lace}({\mathfrak{g}})=\ell \qquad \Longleftrightarrow \qquad {%
\mathfrak{g}}\backslash \ell \subset \mathcal{C}_{N}(\ell ).
\label{char_lace}
\end{equation}%
In particular, \eqref{char_lace} implies that for a fixed lace $\ell \in 
\mathcal{L}_{N},$%
\begin{eqnarray*}
\sum_{{\mathfrak{g}}\colon \mathrm{lace}({\mathfrak{g}})=\ell
}\prod_{(i,j)\in {\mathfrak{g}}}(-U_{ij}) &=&\bigg[\prod\nolimits_{(i,j)\in
\ell }(-U_{ij})\bigg]\sum_{{\mathfrak{g}}\subset \mathcal{C}_{N}(\ell
)}\prod_{(i,j)\in {\mathfrak{g}}}(-U_{ij}) \\
&=&\bigg[\prod\nolimits_{(i,j)\in \ell }(-U_{ij})\bigg]\prod_{(i,j)\in 
\mathcal{C}_{N}(\ell )}(1-U_{ij}).
\end{eqnarray*}%
Now define 
\begin{equation}
\Pi _{N}^{{\scriptsize (k)}}(x):=\frac{\mathrm{d}}{\mathrm{d}x}\sum_{{%
\mathfrak{g}}\in \mathcal{J}_{N}^{{{({k}})}}}\mathbb{E}_{N}\Big[%
\prod\nolimits_{(i,j)\in {\mathfrak{g}}}(-U_{ij}){\mathbf{1}}_{\{B_{N}\in 
\mathrm{d}x\}}\Big],\qquad \mbox{so that}\,\,\Pi _{N}=\sum_{k=1}^{\infty
}\Pi _{N}^{{{({k}})}}.  \label{def-Pi-k}
\end{equation}%
Since $\Pi _{N}^{{{({k}})}}=0$ for large enough $k,$ the above sum is
finite, and 
\[
\sum_{N=1}^{\infty }\lambda ^{N}|\Pi _{N}(x)|\leq \sum_{k=1}^{\infty
}\sum_{N=1}^{\infty }\lambda ^{N}|\Pi _{N}^{{{({k}})}}(x)|.
\]%
We are going to estimate the summands on the right-hand side individually
for every $k$. To control the higher order terms, we first characterise the
structure of a lace. The simple proof of the following lemma is left to the
reader.

\begin{lemma}
\label{lemma-restrict} Any lace 
\[
\ell =\big\{(i_{1},j_{1}),(i_{2},j_{2}),\ldots ,(i_{n},j_{n})\big\}
\]%
with $n\geq 3$ edges, where $i_{1}=1,\ j_{n}=N$ and $i_{k}<j_{k}$ for all $%
k, $ satisfies 
\begin{equation}
1=i_{1}<i_{2}\leq j_{1}<i_{3}\leq j_{2}<\cdots <i_{n-1}\leq
j_{n-2}<i_{n}\leq j_{n-1}<j_{n}=N.  \label{lace_restr}
\end{equation}
\end{lemma}

\begin{remark}
The statement of the above lemma is in disagreement with the corresponding
ones for the standard used lace expansions. In our case, there is the
possibility to have $i_{k+1}=j_{k},$ but there are no other possibilities
for equality. To see the point: In the usually used expansions, one can have
for a lace with $3$ edges (starting with $i_{1}=1$)%
\[
1=i_{1}<i_{2}<j_{1}=i_{3}<j_{2}<j_{3}=N. 
\]%
This is not possible in our case, as $\{(i_{1},j_{1}),(i_{3},j_{3})\}$ would
be irreducible, and therefore it would be the lace in this situation. For
essentially a similar reasoning, $i_{k+1}=j_{k}$ is not possible in the
standard versions, but it is not excluded in our case
\end{remark}

\subsection{Types of Laces and Their Classification.}\label{sec-types-of-laces}

\noindent We need to make a finer disctinction of laces, on base of Lemma %
\ref{lemma-restrict}. We say that two neighboring edges $( i_{k},j_{k} ) ,\
( i_{k+1},j_{k+1} ) $ in the lace \textit{touch }if $j_{k}= i_{k+1}$. Then
we say that $( i_{k+1},j_{k+1} ) $ touches $( i_{k},j_{k} ) $ from the
right. Otherwise, i.e., when $i_{k+1}<j_{k},$ we say that the edges \textit{%
intersect. }

According to Lemma \ref{lemma-restrict}, we classify the laces on $%
\{1,\ldots ,N\}$ of length $n$ into different types, dictated by %
\eqref{lace_restr}. Given a lace $\ell $ with $n$ edges, we order the set of
edges increasingly according to the left end, writing $e_{1},\ldots ,e_{n}$
for the edges. We split the set of edges into disjoint chains%
\begin{equation}
\chi _{j}=(e_{n_{j-1}+1},\ldots ,e_{n_{j}}),\quad j=1,\dots ,k,\qquad %
\mbox{with }0=n_{0}<n_{1}<n_{2}<\dots <n_{k}=n.  \label{chain}
\end{equation}%
Within a chain, and if a chain has at least two elements, $e_{i+1}$ touches $%
e_{i}$ from the right for $n_{j-1}\leq i\leq n_{j}-1$, but $e_{n_{j}+1}$
intersects $e_{n_{j}},$ $j=1,\ldots ,k-1.$ We write%
\[
\tau :=(n_{1},n_{2}-n_{1},\ldots ,n_{k}-n_{k-1}) 
\]%
and call it the \textit{type }of the lace. In case $\tau =(1,\ldots ,1),$
there are no touching edges, and if $\tau =(n),$ the lace is one chain of $n$
touching edges. The set of types is denoted by $\mathfrak{T}_{n}.$ We write $%
\mathcal{L}_{N}^{{{({n,\tau }})}}$ for the set of laces of type $\tau $ on $%
\{1,\ldots ,N\}$ with $n$ edges. We split $\Pi _{N}^{{{({n}})}}$ defined in %
\eqref{def-Pi-k} according to the different types by writing%
\[
\Pi _{N}^{{\scriptsize (n)}}=\sum_{\tau \in \mathfrak{T}_{n}}\Pi _{N}^{%
{\scriptsize (n,\tau )}}, 
\]%
where $\Pi _{N}^{{{({n,\tau }})}}$ is the sum over graphs whose lace have $n$
edges and are of type $\tau $.

Given a lace $\ell $ with $n$ edges, the set of endpoints of its edges can
be written as increasing sequence:%
\begin{equation}
\mathbf{s}(\ell ):=(s_{i})_{i=1,\dots ,m},\qquad \mbox{where }%
1=s_{1}<s_{2}<\cdots <s_{m-1}<s_{m}=N,  \label{def-bf-s}
\end{equation}%
and $m$ depends on $n$ and the type of the lace. If all neighboring edges
intersect, then $m=2n.$ For every touching of neighbor edges, $m$ is reduced
by one.\footnote{%
For instance, if the type is $(2,3,1,2),$ then $n=8$ and $m=12$.}

If we fix $\mathbf{s}(\ell ),$ then we write $\Pi _{N}^{{{({n,\tau ,\mathbf{s%
}}})}}.$

Any $e\in \ell $ is written $e=(s_{\tau _{\star }(e)},s_{\tau ^{\star }(e)})$
with $\tau _{\star }(e)<\tau ^{\star }(e)$. Then $\ell =((s_{\tau _{\star
}(e)},s_{\tau ^{\star }(e)}))_{e\in \ell }$. In an abuse of notation, we
write $\ell $ also as a sequence 
\begin{equation}
((\tau _{\star }(i),\tau ^{\star }(i)))_{i=1,\ldots ,n}  \label{eq-star}
\end{equation}%
where $i$ stands for the $i$-th edge, these being ordered increasingly.%
\footnote{%
For instance if a lace of $n=5$ edges has type type $(1,3,1)$ then the
sequence of pairs $((\tau _{\star }(i),\tau ^{\star }(i)))_{i=1}^{5}$ is $%
((1,3),(2,4),(4,5),(5,7),(6,8))$.}

We define $\ell _{N}^{{\scriptsize (n,\tau ,s)}}$ to be the unique lace of
type $\tau $ that has endpoints $\mathbf{s}(\ell )=\mathbf{s}$ with given $%
\mathbf{s}=(s_{i})_{i=1,\dots ,m}$ as in \eqref{def-bf-s}.

\subsection{Proof of Lemma \protect\ref{Le_1}.}

\noindent With the notation set up in the previous section, we want to
estimate%
\[
\Pi _{N}^{{\scriptsize (n,\tau )}}(y)=\frac{\mathrm{d}}{\mathrm{d}y}%
\sum_{\ell \in \mathcal{L}_{N}^{{{({n,\tau }})}}}\mathbb{E}_{N}\bigg[\bigg(%
\prod_{(i,j)\in \ell }(-U_{ij})\prod_{(i,j)\in \mathcal{C}_{N}(\ell
)}(1-U_{ij})\mathbf{1}_{\{B_{N}\in \mathrm{d}y\}}\bigg].
\]%
\begin{equation}
\lambda ^{N}\Pi _{N}^{{\scriptsize (n,\tau )}}(y)=\sum_{\mathbf{s}}\Big[%
\prod\nolimits_{i=1}^{m}\lambda ^{s_{i}-s_{i-1}}\Big]\Pi _{N}^{{\scriptsize %
(n,\tau ,\mathbf{s})}}(y),  \label{lace_est_1}
\end{equation}%
where%
\[
\Pi _{N}^{{\scriptsize (n,\tau ,\mathbf{s})}}(y):=\frac{\mathrm{d}}{\mathrm{d%
}y}\sum_{\ell \in \mathcal{L}_{N}^{{{({n,\tau ,\mathbf{s}}})}}}\mathbb{E}_{N}%
\bigg[\prod_{(i,j)\in \ell }(-U_{ij})\prod_{(i,j)\in \mathcal{C}_{N}(\ell
)}(1-U_{ij})\mathbf{1}_{\{B_{N}\in \mathrm{d}y\}}\bigg].
\]%
For $N=1$ there are no laces, and $\Pi _{1}=\varphi .$ If $N\geq 2,$ the set
of laces is $\neq \emptyset .$ In the following, we always take $N\geq 2,$ $%
n\geq 1.$ We start with estimating 
\begin{eqnarray}
\left\vert \Pi _{N}^{{\scriptsize (n,\tau ,\mathbf{s})}}(y)\right\vert 
&\leq &\sum_{\ell \in \mathcal{L}_{N}^{{{({n,\tau ,\mathbf{s}}})}}}\frac{%
\mathrm{d}}{\mathrm{d}y}\mathbb{E}_{N}\bigg[\big(\prod\nolimits_{(i,j)\in
\ell }U_{ij}\big)\prod_{j=1}^{m-1}\prod_{s_{j}<u<v\leq s_{j+1}-1}(1-U_{uv}){%
\mathbf{1}}_{\{B_{N}\in \mathrm{d}y\}}\bigg ]  \label{lace_est_2} \\
&=&:g_{N}^{{\scriptsize (n,\tau ,\mathbf{s})}}(y).  \nonumber
\end{eqnarray}%
We represent $g_{N}^{{\scriptsize (n,\tau ,\mathbf{s})}}(y)$ in terms of an
integration over $2m-2$ variables $x_{2},\ldots ,x_{2m-1}\in \mathbb{R}^{d}$
representing the positions of the Brownian motion at times $s_{1}=1,\
s_{2}-1,\ s_{2},\ldots ,\ s_{m}-1=N-1.$ Thus, $s_{1}-1=0$ and also, $B_{0}=0$
and at time $s_{m}=N,$ the Brownian motion is fixed to be at $y$. We define 
\[
\mathbb{P}_{N,\tau ,\mathbf{s}}^{{\scriptsize (x,y)}}(\cdot ):=\mathbb{P}_{N}%
\Big(\cdot \big|B_{s_{i}-1}=x_{2i-1},B_{s_{i}}=x_{2i}\,\,\mathrm{for}%
\,\,i=1,\ldots ,m\Big)
\]%
with the convention $x_{2m}=y,$ and $x_{1}=0.$ This decomposes the path
measure into independent Brownian bridges of time lengths $1$ and $%
s_{i}-1-s_{i-1}$, with $s_{0}=0$. Recall that $\varphi _{t}$ is the Gaussian
density with variance $t$. Then 
\begin{eqnarray}
g_{N}^{{\scriptsize (n,\tau ,\mathbf{s})}}(y) &=&\int_{(\mathbb{R}%
^{d})^{2m-2}}\mathrm{d}\mathbf{x\ }\mathbb{E}_{N,\tau ,\mathbf{s}}^{%
{\scriptsize (\mathbf{x},y)}}\Big[\prod\nolimits_{(i,j)\in \ell
_{N}^{(n,\tau ,\mathbf{t})}}U_{ij}\Big]  \nonumber \\
&&\quad \times \prod_{j=1}^{m-1}\mathbb{E}_{N,\tau ,\mathbf{s}}^{%
{\scriptsize (\mathbf{x},y)}}\Big[\prod\nolimits_{s_{j}<u<v\leq
s_{j+1}-1}(1-U_{uv})\Big]  \label{est-gamma1} \\
&&\quad \times \prod_{j=1}^{m}\varphi
_{1}(x_{2j}-x_{2j-1})\prod_{j=1}^{m-1}\varphi
_{s_{j+1}-1-s_{j}}(x_{2j+1}-x_{2j}).  \nonumber
\end{eqnarray}%
The integration is only over $x_{2},\ldots ,x_{2m-1}$. Also, we have made
the following convention: it is not excluded that $s_{j}=s_{j-1}+1$ in which
case, the conditional law would not be defined, except when $x_{2j}=x_{2j-1}.
$ This is taken care of by interpreting $\varphi _{0}$ as the Dirac function.

Let us now handle factors appearing on the right hand side of %
\eqref{est-gamma1}. First we observe that 
\begin{equation}
\mathbb{E}_{N,\tau ,\mathbf{s}}^{{\scriptsize (\mathbf{x},y)}}\Big[%
\prod\nolimits_{s_{j}<u<v\leq s_{j+1}-1}(1-U_{ij})\Big]\varphi
_{s_{j+1}-1-s_{j}}(x_{2j+1}-x_{2j})=\Gamma
_{s_{j+1}-s_{j}-1}(x_{2j+1}-x_{2j}),  \label{est-gamma2}
\end{equation}%
if $s_{j+1}\geq s_{j}+2.$ In case $s_{j+1}\geq s_{j}+1,$ the product is over
the empty set and therefore equal to $1,$ and so we have to interpret $%
\Gamma _{0}$ as $\delta ,$ too. Actually, also for $s_{j+1}=s_{j}+2,$ the
product is empty, in which case we arrive at $\Gamma _{1}=\varphi _{1}.$

We next express the factor involving the $U_{ij}$ appearing on the
right-hand side of \eqref{est-gamma1} in terms of the function $u_{n}$
defined in \eqref{def-un}. This factor depends on the type $\tau
=(v_{1},\ldots ,v_{k})\in \mathbb{N}^{k}$ of the lace. If $\chi _{1},\ldots
,\chi _{k}$ are the chains (recall \eqref{chain}) then the terms $%
\prod\nolimits_{(i,j)\in \chi _{r}}U_{ij},$ $r=1,\ldots ,k$ are independent
under the conditional measure $\mathbb{P}_{N,\tau ,\mathbf{s}}^{{\scriptsize %
(\mathbf{x},y)}}.$ Let a chain $\chi $ contain $b$ edges, labelled $%
(s_{i_{1}},s_{i_{2}}),\ (s_{i_{2}},s_{i_{3}}),\ldots ,(s_{i_{b}},s_{i_{b+1}})
$. Then $i_{k+1}-i_{k}=1,2$ or $3,$ depending whether the chain is in
relation to the others. In any case, if $b\geq 3,$ then $i_{k+1}-i_{k}=1$
for $k=2,\ldots ,b-1$. Then by recalling the definition of $u_{n}(\cdot )$
from \eqref{def-un}, 
\begin{equation}
\mathbb{E}_{N,\tau ,\mathbf{s}}^{{\scriptsize (\mathbf{x},y)}}\Big[%
\prod\nolimits_{(i,j)\in \chi }U_{ij}\Big]\prod_{k=1}^{b+1}\varphi
_{1}(x_{i_{k}}-x_{i_{k}-1})=u_{b+1}(x_{i_{1}-1},x_{i_{1}},x_{i_{2}-1},x_{i_{2}},\ldots ,x_{i_{b+1}-1},x_{i_{b+1}}).
\label{est-gamma3}
\end{equation}%
By Lemma \ref{Le_Est_u} we obtain 
\begin{eqnarray*}
&&u_{b+1}(x_{i_{1}-1},x_{i_{1}},\ldots ,x_{i_{b+1}-1},x_{i_{b+1}}) \\
&\leq &C^{b+1}\alpha ^{b}\exp \Big[-\frac{1}{8}\sum%
\nolimits_{k=1}^{b+1}|x_{i_{k}-1}-x_{i_{k}}|^{2}-A%
\sum_{k=1}^{b}|x_{i_{k+1}}-x_{i_{k}}|^{2}\Big].
\end{eqnarray*}

Combining \eqref{est-gamma1}-\eqref{est-gamma3} we therefore get%
\begin{eqnarray*}
g_{N}^{{\scriptsize (n,\tau ,\mathbf{s})}}(y) &\leq &C^{n}\alpha ^{n}\int 
\mathrm{d}\mathbf{x}{\mathbf{\ }}\prod_{j=1}^{m-1}\Gamma
_{s_{j+1}-1-s_{j}}(x_{2j+1}-x_{2j}) \\
&&\quad \times \exp \Big[-\frac{1}{8}\sum%
\nolimits_{j=1}^{m}|x_{2j}-x_{2j-1}|^{2}-A\sum\nolimits_{i=1}^{n}|x_{2\tau
^{\star }(i)}-x_{2\tau _{\star }(i)}|^{2}\Big],
\end{eqnarray*}%
where $A>0$ does not depend on $\alpha ,n,$ or $\tau ,$ and in the last
display we used the notation $\tau ^{\star }(i)$ and $\tau _{\star }(i)$
defined in \eqref{eq-star}.

From the preceding, we get, for any type $\tau $ and any $\mathbf{s}$,
recalling that $s_{1}=1,\ s_{m}=N,$ 
\begin{eqnarray*}
\sum_{N=2}^{\infty }\lambda ^{N}\sum_{\ell \in \mathcal{L}_{N}^{\left(
n,\tau \right) }}g_{N}^{{\scriptsize (n,\tau ,\mathbf{s})}}(y) &\leq
&C^{n}\alpha ^{n}\lambda ^{m-1}\int \mathrm{d}{\mathbf{x\ }}%
\prod_{j=1}^{m-1}\sum_{k=0}^{\infty }\lambda ^{k}\Gamma _{k}(x_{2j-1}-x_{2j})
\\
&&\quad \times \exp \Big[-\frac{1}{8}\sum%
\nolimits_{j=1}^{m}|x_{2j}-x_{2j-1}|^{2}\Big]\exp \Big[-A\sum%
\nolimits_{i=1}^{n}|x_{2\tau ^{\star }(i)}-x_{2\tau _{\star }(i)}|^{2}\Big]
\\
&\leq &C^{n}\alpha ^{n}\int \mathrm{d}{\mathbf{x}}\prod_{j=1}^{m-1}\Big[%
\delta (x_{2j+1}-x_{2j})+G_{\lambda }^{{\scriptsize (\Gamma )}%
}(x_{2j+1}-x_{2j})\Big] \\
&&\quad \times \mathrm{e}^{-\left\vert x_{2}\right\vert ^{2}/8}\exp \Big[-%
\frac{1}{8}\sum\nolimits_{j=1}^{m-1}|x_{2j+2}-x_{2j+1}|^{2}\Big] \\
&&\quad\times \exp \Big[%
-A\sum\nolimits_{i=1}^{n}|x_{2\tau ^{\star }(i)}-x_{2\tau _{\star }(i)}|^{2}%
\Big],
\end{eqnarray*}%
where we recall that $m\leq 2n.$ We integrate now over all the $x_{2j+1},\
1\leq j\leq m-1$. Using the assumption $G_{\lambda }^{{(}{\scriptsize \Gamma
)}}\leq 7G$ from \eqref{3_bound} of Lemma \ref{Le_1}, we get%
\begin{equation}
\int \mathrm{d}x_{2j+1}\bigg[\delta (x_{2j+1}-x_{2j})+G_{\lambda }^{%
{\scriptsize (\Gamma )}}(x_{2j+1}-x_{2j})\bigg]\exp \Big[-\frac{1}{8}%
|x_{2j+2}-x_{2j}|^{2}\Big]\leq Cf_{d}(|x_{2j+2}-x_{2j}|),  \label{est-gamma6}
\end{equation}%
where we abbreviated 
\begin{equation}
f_{d}(x):=(1+|x|)^{-d+2}.  \label{def-fd}
\end{equation}

Combining the last two estimates we obtain%
\begin{equation}
\sum_{N=2}^{\infty }\lambda ^{N}\sum_{\ell \in \mathcal{L}_{N}^{\left(
n,\tau \right) }}g_{N}^{\left( n,\tau ,\mathbf{s}\right) }\left( y\right)
\leq C^{n}\alpha ^{n}\left( \varphi _{1/4}\ast X_{n,\tau }\right) \left(
y\right) ,  \label{g_N_est}
\end{equation}%
where for $n\geq 1$, with renaming $z_{i}:=x_{2i},$ we have written%
\begin{equation}\label{X_n_n>1} 
\begin{aligned}
X_{n,\tau }\left( x\right)  &:=\int\dots \int \d z_{2}\cdots \d z_{m-1} \,  \exp \bigg[-A\sum\nolimits_{i=1}^{n-1}\left\vert z_{\tau
^{\ast }\left( i\right) }-z_{\tau _{\ast }\left( i\right) }\right\vert ^{2}%
\bigg]\exp \bigg[-A\left\vert z_{\tau _{\ast }\left( n\right) }-x\right\vert
^{2}\bigg]  \\
&\qquad \times f_{d}\left( x-z_{m-1}\right) \prod_{j=2}^{m-1}f_{d}\left(
z_{j}-z_{j-1}\right) , 
\end{aligned}
\end{equation}
where $z_{1}:=0.$ We also remark that $m$ in the expression above is a
function of $\tau $ that satisfies $m\left( \tau \right) \leq 2n.$ Also $m=2$
for $n=1,$ and $m\geq 3$ for $n\geq 2.$

We will prove by induction on $n$ that for some constant $C,$ depending only
on $A$, and $d,$ we have, with $F_{d}\left( x\right) :=f_{d}\left( x\right)
^{3}=(1+|x|)^{-3d+6}$ 
\begin{equation}
\sup_{\tau \in \mathfrak{T}_{n}}X_{n,\tau }\left( x\right) \leq
C^{n}F_{d}\left( x\right) ,\qquad n\in \mathbb{N}.  \label{main_bound_L1}
\end{equation}%
We remark that from this, it also follows that%
\[
\sup_{\tau \in \mathfrak{T}_{n}}\left( \varphi _{1/4}\ast X_{n,\tau }\right)
\left( x\right) \leq C^{n}F_{d}\left( x\right) ,\qquad n\in \mathbb{N},
\]%
with a slightly changed constant $C$. Implementing this into (\ref{g_N_est}%
), and as the number of types of laces with $n$ edges is bounded by $2^{n},$
this will prove Lemma \ref{Le_1}.

It remains to prove \eqref{main_bound_L1} by induction. For $n=1$, there is
only one type, and one simply has%
\begin{equation}
X_{1}\left( x\right) =\exp \big[-A\left\vert x\right\vert ^{2}\big]%
f_{d}\left( x\right)  \label{X_1}
\end{equation}%
so (\ref{main_bound_L1}) is evident for $n=1.$

Let $n=2.$ Then there are the two types $\left( 1,1\right) $ and $(2).$ For $%
\tau =\left( 2\right) $%
\begin{eqnarray*}
X_{2}\left( x\right)  &=&\int \mathrm{d}\mathbf{z}\exp \big[-A\left\vert
z\right\vert ^{2}\big]f_{d}\left( z\right) \exp \big[-A\left\vert
x-z\right\vert ^{2}\big]f_{d}\left( x-z\right)  \\
&\leq &\int \mathrm{d}\mathbf{z}\exp \big[-A\left\vert z\right\vert
^{2}-A\left\vert x-z\right\vert ^{2}\big],
\end{eqnarray*}%
which has Gaussian tails, so (\ref{main_bound_L1}) is gain trivial. So we
have only to take care of $\tau =\left( 1,1\right) .$ There we have $m=4,$
and%
\[
X_{2}\left( x\right) =\int \mathrm{d}z_{1}\mathrm{d}z_{2}\,\,\exp \big[-A%
\big(\left\vert z_{2}\right\vert ^{2}+\left\vert x-z_{1}\right\vert ^{2}\big)%
\big]\,\,f_{d}\left( z_{1}\right) f_{d}\left( z_{2}-z_{1}\right) f_{d}\left(
x-z_{2}\right) ,
\]%
which, according to Lemma \ref{Le_Gauss_convol}, is bounded by $C\left(
1+\left\vert x\right\vert \right) ^{-3d+6}=CF_{d}\left( x\right) .$

We assume now $n\geq 3,$ and that \eqref{main_bound_L1} holds up to $n-1$
instead of $n$. We consider first the case where the type $\tau $ starts
with $\left( 1,1,\ldots \right) $ i.e. $\tau =\left( 1,1,\tau ^{\prime
}\right) .$ Then the expression for $X_{n,\tau }$ starts with%
\begin{eqnarray*}
&&\int \mathrm{d}\mathbf{z}\ \exp \big[-A\big(\left\vert z_{2}\right\vert
^{2}+\left\vert z_{4}-z_{1}\right\vert ^{2}+\left\vert
z_{5}-z_{3}\right\vert ^{2}+\cdots \big)\big] \\
&&\qquad \times f_{d}\left( z_{1}\right) f_{d}\left( z_{2}-z_{1}\right)
f_{d}\left( z_{3}-z_{2}\right) f_{d}\left( z_{4}-z_{3}\right) f_{d}\left(
z_{5}-z_{4}\right) \cdot \cdots
\end{eqnarray*}%
or $\left\vert z_{5}-z_{3}\right\vert ^{2}$ replaced by $\left\vert
z_{6}-z_{3}\right\vert ^{2}$ if $n>3$ and the third entry in $\tau $ is $%
\neq 1.$

If the type starts with $\left( 1,1,...\right) $ as assumed here, then the
first three bonds overlap. This situation is similar to the one in the
standard version of the lace expansion. The only difference is the presence
of the factor $\exp \left[ -A\left\vert z_{2}\right\vert ^{2}-A\left\vert
z_{4}-z_{1}\right\vert ^{2}-\cdots \right] $ which replaces the fixed
matching of the endpoints that appears on the standard lattice situation. We
proceed now to remove the second bound in the lace in a standard way, using
Cauchy-Schwarz, and $d\geq 5:$ 
\begin{eqnarray}
&&\int \mathrm{d}z_{1}\mathrm{d}z_{4}\,f_{d}\left( z_{1}\right) f_{d}\left(
z_{2}-z_{1}\right) f_{d}\left( z_{3}-z_{2}\right) f_{d}\left(
z_{4}-z_{3}\right) f_{d}\left( z_{5}-z_{4}\right) \exp \left[ -A\left\vert
z_{4}-z_{1}\right\vert ^{2}\right]  \nonumber \\
&\leq &\sqrt{\int f_{d}^{2}\left( z_{1}\right) f_{d}^{2}\left(
z_{4}-z_{3}\right) \exp \left[ -A\left\vert z_{4}-z_{1}\right\vert ^{2}%
\right] \,\mathrm{d}z_{1}\mathrm{d}z_{4}}  \nonumber \\
&&\times \sqrt{\int f_{d}^{2}\left( z_{2}-z_{1}\right) f_{d}^{2}\left(
z_{5}-z_{4}\right) \exp \left[ -A\left\vert z_{4}-z_{1}\right\vert ^{2}%
\right] \,\mathrm{d}z_{1}\mathrm{d}z_{4}}  \label{Laces_by_CS} \\
&\leq &C\sqrt{\int f_{d}^{2}\left( z_{1}\right) f_{d}^{2}\left(
z_{1}-z_{3}\right) \,\mathrm{d}z_{1}}\sqrt{\int f_{d}^{2}\left(
z_{2}-z_{1}\right) f_{d}^{2}\left( z_{5}-z_{1}\right) \,\mathrm{d}z_{1}} 
\nonumber \\
&\leq &Cf_{d}\left( z_{3}\right) f_{d}\left( z_{5}-z_{2}\right) .  \nonumber
\end{eqnarray}%
In the first inequality, we drop $f_{d}\left( z_{3}-z_{2}\right) \leq 1.$ In
the second, we use 
\begin{equation}
\int f_{d}^{2}\left( x-y\right) \exp \left[ -A\left\vert y\right\vert ^{2}%
\right] \mathrm{d}y\leq Cf_{d}^{2}\left( x\right) .  \label{fd_est}
\end{equation}%
Below, we need a tiny bit of information when $A$ is replace by $A/k,\ k\in 
\mathbb{N}$. It's easy to see that%
\begin{equation}
\int f_{d}^{2}\left( x-y\right) \exp \left[ -\left( A/k\right) \left\vert
y\right\vert ^{2}\right] \mathrm{d}y\leq \theta ^{k}f_{d}^{2}\left( x\right)
\label{modified_fd_est}
\end{equation}%
for some $\theta >1$ (depending also on $A$).

To justify the last step in (\ref{Laces_by_CS}), we note that $%
f_{d}^{2}(x)=(1+|x|)^{4-2d}$ is integrable in $d\geq 5$. Then the first
integral of the fourth line above is a convolution of $f_{d}^{2}$ with
itself, and we can take Fourier transform of this convolution, use a uniform
bound on the Fourier coefficient of $f_{d}^{2}$ and then take inverse
Fourier transform to obtain the desired bound.

Also, if $n>3$ and the third entry in $\tau $ is $\neq 1,$ $z_{5}$ is
replaced by $z_{6}.$ Implementing this inequality then leads to%
\begin{equation}
X_{n,\tau }\leq CX_{n-1,\left( 1,\tau ^{\prime }\right) }.  \label{Ind_Est}
\end{equation}

We next consider the case where $\tau :=\left( 1,t_{2},t_{3},\ldots \right) $
with $t_{2}>1.$ \ Here only the first two bonds of the lace overlap, and the
third touches the second. \ Assume that the bonds number $2,3,...,k+1$ have
the property that bond number $j,\ 3\leq j\leq k+1$ touches bond number $j-1$
from the right, i.e. $t_{2}=k.$ Then the expression we have to integrate (in
order to get $X_{n,\tau }\left( y\right) $ starts with%
\begin{eqnarray*}
&&\exp \left[ -A\left( \left\vert z_{2}\right\vert ^{2}+\left\vert
z_{3}-z_{1}\right\vert ^{2}+\sum\nolimits_{j=3}^{k}\left\vert
z_{j+1}-z_{j}\right\vert ^{2}+\left\vert z_{k+3}-z_{k+1}\right\vert
^{2}+\left\vert z_{k+4}-z_{k+2}\right\vert ^{2}\cdots \right) \right]  \\
&&\times f_{d}\left( z_{1}\right) \cdot f_{d}\left( z_{2}-z_{1}\right) \cdot
\cdots \cdot f_{d}\left( z_{k+2}-z_{k+1}\right) \cdot f_{d}\left(
z_{k+3}-z_{k+2}\right) \cdot \cdots .
\end{eqnarray*}%
We get an upper bound by dropping $f_{d}\left( z_{3}-z_{2}\right) ,$ $%
f_{d}\left( z_{4}-z_{3}\right) $ up to $f_{d}\left( z_{k+2}-z_{k+1}\right) .$
Then we integrate over $z_{3},...,z_{k+1}.$ These variables appear then only
in the Gaussian factor. Integrating that out, we get an upper bound for $%
\int \mathrm{d}z_{3}\cdots \mathrm{d}z_{k+1}$ of the above expression as%
\begin{eqnarray*}
&&\left( \pi /A\right) ^{\left( k-1\right) d/2}k^{-d/2}\exp \left[
-A\left\vert z_{2}\right\vert ^{2}-\frac{A}{k}\left\vert
z_{k+3}-z_{1}\right\vert ^{2}-A\left\vert z_{k+4}-z_{k+2}\right\vert
^{2}-\cdots \right]  \\
&&\times f_{d}\left( z_{1}\right) \cdot f_{d}\left( z_{2}-z_{1}\right) \cdot
f_{d}\left( z_{k+3}-z_{k+2}\right) \cdot \cdots 
\end{eqnarray*}%
This is the same expression as the first in (\ref{Laces_by_CS}) (with
already dropped $f_{d}\left( z_{3}-z_{2}\right) $), with $%
z_{3},z_{4},z_{5},...$ there replaced by $z_{k+2},z_{k+3},z_{k+4},....$
Furthermore, instead of $-A\left\vert z_{4}-z_{1}\right\vert $ in the
exponent, we have now $-\left( A/k\right) \left\vert
z_{k+3}-z_{1}\right\vert ^{2}.$ But we have only to use (\ref%
{modified_fd_est}) instead of (\ref{fd_est}), and in this way get 
\[
X_{n,\tau }\leq C^{t_{2}}X_{n-t_{2},\left( 1,t_{3},...\right) }.
\]

Finally, we have to handle $\tau =\left( k,t_{2},\ldots \right) $ with $%
k\geq 2.$ We first look at the case $\tau =(k)$ in which case $k=n$ if $\tau
\in \mathfrak{T}_{n}.$ Then%
\begin{eqnarray*}
X_{n,\tau }\left( x\right)  &=&\int \mathrm{d}\mathbf{z}\exp \bigg[-A\big(%
\left\vert z_{1}\right\vert ^{2}+\left\vert z_{2}-z_{1}\right\vert
^{2}+\cdots +\left\vert x-z_{n-1}\right\vert \big)\bigg] \\
&&\qquad \times f_{d}\left( z_{1}\right) f_{d}\left( z_{2}-z_{1}\right)
\cdot \cdots \cdot f_{d}\left( x-z_{n-1}\right)  \\
&\leq &\int \mathrm{d}\mathbf{z}\exp \bigg[-A\big(\left\vert
z_{1}\right\vert ^{2}+\left\vert z_{2}-z_{1}\right\vert ^{2}+\cdots
+\left\vert x-z_{n-1}\right\vert \big)\bigg] \\
&\leq &c^{n}\exp \bigg[-\frac{A}{n}\left\vert x\right\vert ^{2}\bigg]
\end{eqnarray*}%
for a constant $c>1$ (depending on $A,d$ only). It is easily checked that $%
\exp \big[-\frac{A}{n}\left\vert x\right\vert ^{2}\big]\leq \left( c^{\prime
}\right) ^{n}F_{d}\left( x\right) $ for some constant $c^{\prime }=c^{\prime
}\left( A\right) >0.$ (actually, $n^{\gamma }F_{d}\left( x\right) $ for some 
$\gamma >0$ is already an upper bound). So \eqref{main_bound_L1} is
satisfied for $X_{n,(n)}$ for all $n,$ if $C$ is chosen appropriately.

Let now $\tau =\left( k,\tau ^{\prime }\right) $ with $k\geq 2.$ Take the
types $\left( k-1\right) \in \mathfrak{T}_{k-1}$ and $\left( 1,\tau ^{\prime
}\right) \in \mathfrak{T}_{n-k+1}.$ Observe that%
\[
X_{n,\left( k,\tau ^{\prime }\right) }=X_{k-1,\left( k-1\right) }\ast
X_{n-k+1,\left( 1,\tau ^{\prime }\right) }.
\]%
As $k\geq 2,$ we can use the induction assumption for the second factor and
obtain%
\[
X_{n,\left( k,\tau ^{\prime }\right) }\leq C_{1}^{k-1}C^{n-k+1}\left(
F_{d}\ast F_{d}\right) \leq \gamma C_{1}^{k-1}C^{n-k+1}F_{d},
\]%
where we use $F_{d}\ast F_{d}\leq \gamma F_{d}$ for some constant $\gamma
=\gamma _{d}>0$ (because $3d-6>d+1$). Therefore, if $\gamma C_{1}\leq C$, we
have%
\[
X_{n,\left( k,\tau ^{\prime }\right) }\leq C^{n}F_{d},
\]%
finishing the induction argument for \eqref{main_bound_L1}. Therefore, Lemma %
\ref{Le_1} is proved.

\section{Deconvolution: Proof of Lemma \protect\ref{Le_2}\label{Le_2_proof}}

Consider the separable Banach space $\left( \mathbf{B},\left\Vert \cdot
\right\Vert \right) $ of bounded continuous functions $f:\mathbb{R}%
^{d}\rightarrow \mathbb{R}$ which have finite norm%
\[
\left\Vert f\right\Vert :=\max \left( \left\Vert f\right\Vert
_{1},\left\Vert f\right\Vert _{\infty },\sup\nolimits_{x}\left\vert
x\right\vert ^{d}\left\vert f\left( x\right) \right\vert \right) <\infty 
\]

\begin{lemma}\label{lem-normesti}
If $f,g\in \mathbf{B}$ then also $f\ast g\in \mathbf{B}.$ Furthermore%
\[
\left\Vert f\ast g\right\Vert \leq 2^{d+1}\left\Vert f\right\Vert \left\Vert
g\right\Vert 
\]
\end{lemma}

\begin{proof}
It is clear that for $f,g\in \mathbf{B}$, also $f\ast g$ is continuous.
Furthermore $\left\Vert f\ast g\right\Vert _{1}=\left\Vert f\right\Vert
_{1}\left\Vert g\right\Vert _{1}\leq \left\Vert f\right\Vert \left\Vert
g\right\Vert ,$ and $\left\Vert f\ast g\right\Vert _{\infty }\leq \left\Vert
f\right\Vert _{1}\left\Vert g\right\Vert _{\infty }\leq \left\Vert
f\right\Vert \left\Vert g\right\Vert .$ Concerning the third part in the
definition of $\left\Vert \cdot \right\Vert ,$ let $x\in \mathbb{R}^{d}.$
Then%
\begin{eqnarray*}
\left\vert \left( f\ast g\right) \left( x\right) \right\vert &=&\left\vert
\int f\left( y\right) g\left( x-y\right) dy\right\vert \leq \int \left\vert
f\left( y\right) \right\vert \left\vert g\left( x-y\right) \right\vert dy \\
&\leq &\int_{y:\left\vert y\right\vert >\left\vert x-y\right\vert
}\left\vert f\left( y\right) \right\vert \left\vert g\left( x-y\right)
\right\vert dy+\int_{y:\left\vert y\right\vert \leq \left\vert
x-y\right\vert }\left\vert f\left( y\right) \right\vert \left\vert g\left(
x-y\right) \right\vert dy.
\end{eqnarray*}%
\[
\left\{ y:\left\vert x-y\right\vert <\left\vert y\right\vert \right\}
\subset \left\{ y:\left\vert y\right\vert >\left\vert x\right\vert
/2\right\} , 
\]%
and%
\[
\left\{ y:\left\vert x-y\right\vert \geq \left\vert y\right\vert \right\}
\subset \left\{ y:\left\vert x-y\right\vert \geq \left\vert x\right\vert
/2\right\} . 
\]%
Therefore,%
\begin{eqnarray*}
\left\vert x\right\vert ^{d}\left\vert \left( f\ast g\right) \left( x\right)
\right\vert &\leq &\left\vert x\right\vert ^{d}\int_{y:2^{d}\left\vert
y\right\vert ^{d}>\left\vert x\right\vert ^{d}}\left\vert f\left( y\right)
\right\vert \left\vert g\left( x-y\right) \right\vert dy \\
&&+\left\vert x\right\vert ^{d}\int_{y:2^{d}\left\vert x-y\right\vert
^{d}\geq \left\vert x\right\vert ^{d}}\left\vert f\left( y\right)
\right\vert \left\vert g\left( x-y\right) \right\vert dy \\
&\leq &2^{d}\left( \sup\nolimits_{y}\left\vert y\right\vert ^{d}\left\vert
f\left( y\right) \right\vert \right) \left\Vert g\right\Vert _{1} \\
&&+2^{d}\left( \sup\nolimits_{y}\left\vert y\right\vert ^{d}\left\vert
g\left( y\right) \right\vert \right) \left\Vert f\right\Vert _{1} \\
&\leq &2^{d+1}\left\Vert f\right\Vert \left\Vert g\right\Vert
\end{eqnarray*}%
This proves the claim.
\end{proof}

We remind the reader on \eqref{Def_G}: $G(x)=\sum_{n=1}^{\infty }\varphi
_{n}(x)$ is the Green function for the random walk with standard normal
increments.

\begin{remark}
\label{Rem_Lemma2}

\begin{enumerate}
\item[a)] Remark that \eqref{Bound_Laplace_Pi} implies 
\begin{equation}
\left\Vert G_{\lambda }^{\Pi }\right\Vert \leq C  \label{GPi_bound}
\end{equation}%
and%
\[
\left\Vert G_{\lambda }^{\Pi }-\lambda \varphi \right\Vert \leq C\alpha . 
\]

\item[b)] If a function $f$ satisfies $\left\Vert f\right\Vert \leq 1,$ then 
$\left\vert f\left( x\right) \right\vert \leq CG\left( x\right) $
\end{enumerate}
\end{remark}

\begin{lemma}\label{lemma-rho}
Assume that $\rho :\mathbb{R}^{d}\rightarrow \mathbb{R}$ is a rotational
invariant continuous function satisfying 
\[
\left\vert \rho \left( x\right) \right\vert \leq \left( 1+\left\vert
x\right\vert \right) ^{-d-4},\ \int \rho \left( x\right) dx=0. 
\]%
Then%
\[
\left\Vert \rho \ast G\right\Vert \leq C<\infty 
\]
\end{lemma}

\begin{proof}
$G,\rho $ are both bounded and continuous, and $\rho $ is integrable.
Therefore $\rho \ast G$ is continuous, and satisfies $\left\Vert \rho \ast
G\right\Vert _{\infty }\leq C<\infty .$ If we check that $\left\vert
x\right\vert ^{d}\left\vert \left( \rho \ast G\right) \left( x\right)
\right\vert \leq C<\infty $, we are done. It suffices to prove this for $%
\left\vert x\right\vert \geq 1$ which we assume for the rest of the proof.
Evidently%
\[
\left( \rho \ast G\right) \left( x\right) =\int_{\left\vert y\right\vert
\geq 1}G\left( y\right) \rho \left( x-y\right) \mathrm{d}y+O\left(
\left\vert x\right\vert ^{-d-4}\right) . 
\]%
We use now that for $\left\vert y\right\vert \geq 1,$%
\[
G\left( y\right) =a\left\vert y\right\vert ^{-d+2}+h\left( y\right) 
\]%
with $\left\vert h\left( y\right) \right\vert \leq C\left( 1+\left\vert
y\right\vert \right) ^{-d-2}$ (Lemma \ref{Le_Edgeworth}). First:%
\begin{eqnarray*}
\int h\left( y\right) \rho \left( x-y\right) dy &=&\int\limits_{\left\vert
y\right\vert \leq \left\vert x\right\vert /2}h\left( y\right) \rho \left(
x-y\right) dy+\int\limits_{\left\vert y\right\vert >\left\vert x\right\vert
/2}h\left( y\right) \rho \left( x-y\right) dy \\
&\leq &\left( 1+\left\vert x\right\vert \right) ^{-d-4}\left\Vert
h\right\Vert _{1}+\left( 1+\left\vert x\right\vert \right) ^{-d-2}\left\Vert
\rho \right\Vert _{1} \\
&\leq &C\left( 1+\left\vert x\right\vert \right) ^{-d-2}.
\end{eqnarray*}%
Therefore, we have to care only for%
\begin{eqnarray*}
\int \rho \left( y\right) \left\vert x-y\right\vert ^{-d+2}dy
&=&\int\limits_{\left\vert y\right\vert \leq \left\vert x\right\vert /2}\rho
\left( y\right) \left\vert x-y\right\vert ^{-d+2}dy+\int\limits_{\left\vert
y\right\vert >\left\vert x\right\vert /2}\rho \left( y\right) \left\vert
x-y\right\vert ^{-d+2}dy \\
&=&:I_{1}+I_{2}.
\end{eqnarray*}%
In the first summand $I_{1}$, we Taylor expand $\left\vert x-y\right\vert
^{-d+2}$ around $x:$%
\begin{eqnarray*}
\left\vert x-y\right\vert ^{-d+2} &=&\left\vert x\right\vert
^{-d+2}+\sum_{i=1}^{d}y_{i}\left( 2-d\right) x_{i}\left\vert x\right\vert
^{-d} \\
&&+\sum_{i,j=1}^{d}y_{i}y_{j}\left( -\left( 2-d\right) d\cdot
x_{i}x_{j}\left\vert x\right\vert ^{-d-2}+\delta _{ij}\left( 2-d\right)
\left\vert x\right\vert ^{-d}\right) +O\left( \left\vert y\right\vert
^{3}\left\vert x\right\vert ^{-d-1}\right) .
\end{eqnarray*}%
We split $I_{1}$ accordingly into $I_{1}=I_{11}+I_{12}+I_{13}+I_{14}.$

For $I_{11}$ we use that $\int \rho \left( x\right) dx=0$ which leads to%
\[
\left\vert I_{11}\right\vert \leq \left\vert x\right\vert
^{-d+2}\int_{\left\vert y\right\vert \geq \left\vert x\right\vert /2}\rho
\left( y\right) dy\leq C\left\vert x\right\vert ^{-d-2}. 
\]%
By the symmetry of $\rho ,$ we have $I_{12}=0,$ and similarly for the
off-diagonal terms $i\neq j$ in $I_{13}.$ For the on-diagonal terms we use
that $\int y_{i}^{2}\rho \left( y\right) dy$ does not depend on $i,$ and
that by the harmonicity of $\left\vert x\right\vert ^{-d+2}$ in $\mathbb{R}%
^{d}-\left\{ 0\right\} $%
\[
\sum_{i=1}^{d}\left[ -\left( 2-d\right) d\cdot x_{i}^{2}\left\vert
x\right\vert ^{-d-2}+\left( 2-d\right) \left\vert x\right\vert ^{-d}\right]
=0 
\]%
Therefore $I_{13}=0.$ As for $I_{14},$ we have%
\[
\left\vert I_{14}\right\vert \leq \int_{\left\vert y\right\vert \leq
\left\vert x\right\vert /2}\left\vert \rho \left( y\right) \right\vert
\left\vert y\right\vert ^{3}\left\vert x\right\vert ^{-d-1}=O\left(
\left\vert x\right\vert ^{-d-1}\right) ,\ \left\vert x\right\vert \geq 1. 
\]%
The estimate of $I_{2}$ is simpler. Using the decay of $\rho ,$ one easily
checks that%
\[
\left\vert I_{2}\right\vert \leq C\left( 1+\left\vert x\right\vert \right)
^{-d-2} 
\]%
which is better than required.\hfill $\diamondsuit $
\end{proof}

\begin{lemma}
\label{Le_rho2}Let $\mu \in \left[ 0,1\right] $ and define%
\[
G_{\mu }(x):=\sum_{n=1}^{\infty }\mu ^{n}\varphi _{n}(x),\qquad x\in \mathbb{%
R}^{d}. 
\]%
Then, if $\mu <1,$ then $G_{\mu }\in \mathbf{B}$, and if $\rho $ satisfies
the conditions of Lemma \ref{lemma-rho}, then 
\begin{equation}
\Vert \rho \star G_{\mu }\Vert \leq C\mu ,\qquad \mu \in \left[ 0,1\right] .
\label{Est_rho_G_mu}
\end{equation}
\end{lemma}

\begin{proof}
Let us first verify that $G_{\mu }\in \mathbf{B}$ for any $\mu \in \lbrack
0,1)$. Since $\left\Vert \varphi _{n}\right\Vert =\varphi _{n}(0)\leq \left(
2\pi n\right) ^{-d/2}\leq 1=\left\Vert \varphi _{n}\right\Vert _{1}$, we
have $\Vert G_{\mu }\Vert _{\infty },\ \leq \Vert G_{\mu }\Vert _{1}\leq 
\frac{\mu }{1-\mu }$. \ Furthermore, 
\[
\sup_{x\in \mathbb{R}^{d}}|x|^{d}G_{\mu }(x)\leq \sum_{n=1}^{\infty }\mu ^{n}%
\frac{1}{(2\pi n)^{d/2}}\,\sup_{x\in \mathbb{R}^{d}}|x|^{d}\mathrm{e}%
^{-|x|^{2}/2n}=d^{d/2}\mathrm{e}^{-d/2}\sum_{n=1}^{\infty }\mu ^{n}\frac{1}{%
(2\pi n)^{d/2}}n^{d/2}\leq \frac{C\mu }{1-\mu }. 
\]%
Hence, $G_{\mu }\in \mathbf{B}$ and $\Vert G_{\mu }\Vert \leq \frac{C\mu }{%
1-\mu }$.

For (\ref{Est_rho_G_mu}), we can assume $\mu \neq 1.$ An elementary
calculation shows that 
\[
\rho \star G_{\mu }=\mu (\rho \star G)-(1-\mu )\rho \star G\star G_{\mu }. 
\]%
Therefore, for any $\mu \in \lbrack 0,1)$, 
\[
\Vert \rho \star G_{\mu }\Vert \leq \mu \Vert \rho \star G\Vert +(1-\mu
)2^{d+1}\Vert \rho \star G\Vert \,\Vert G_{\mu }\Vert \leq C\mu \Vert \rho
\star G\Vert 
\]%
by applying Lemma~\ref{lem-normesti}. Now Lemma~\ref{lemma-rho} implies the
assertion.
\end{proof}

\subsection{Proof of Lemma \protect\ref{Le_2}.}

We choose%
\begin{equation}
\mu :=\int G_{\lambda }^{{\scriptsize (\Pi )}}\left( x\right) dx.
\label{Def_mu}
\end{equation}%
Remark that from (\ref{Basic_Green_identity}), we have%
$$
\left( 1-\mu \right) \left( 1+\int G_{\lambda }^{{\scriptsize (\Gamma )}%
}\left( x\right) dx\right) 
=\left( 1-\int G_{\lambda }^{{\scriptsize (\Pi )}}\left( x\right)
dx\right) \left( 1+\int G_{\lambda }^{{\scriptsize (\Gamma )}}\left(
x\right) dx\right) =1.
$$%
As $G_{\lambda }^{\left( \Gamma \right) }\geq 0,$ the second factor above is 
$\geq 1,$ and therefore the first is in $(0,1],$ implying $
0\leq \mu <1$.
By assumption $G_{\lambda }^{\left( \Pi \right) }$ satisfies (\ref%
{Bound_Laplace_Pi}), we get also $
\left\vert \mu -\lambda \right\vert \leq C\alpha$.
With this choice of $\mu ,$ we have $\int \rho \left( x\right) \mathrm{d}x=0$
for%
\[
\rho :=-G_{\lambda }^{{\scriptsize (\Pi )}}+\mu \varphi ,
\]%
and, as $d\geq 5,$ and therefore $3d-6\geq d+4$,%
$$
\left\vert \rho \left( x\right) \right\vert  \leq \left\vert G_{\lambda }^{%
{\scriptsize (\Pi )}}\left( x\right) \right\vert +\alpha \varphi \left(
x\right)\leq C\alpha \left( 1+\left\vert x\right\vert \right) ^{-d-4}.
$$%
In particular, this implies that%
\begin{equation}
\left\Vert \rho \right\Vert \leq C\alpha ,  \label{bound_3}
\end{equation}%
and using Lemma \ref{Le_rho2} (applied to $\rho /C\alpha $ instead of $\rho $%
)
\[
\left\Vert \rho \ast G_{\mu }\right\Vert \leq C\alpha \left\vert \mu
\right\vert \leq C\alpha .
\]%
Using (\ref{bound_3}), this also implies%
\begin{equation}
\left\Vert \rho \ast G_{\mu }+\rho \right\Vert \leq C\alpha .
\label{bound_4}
\end{equation}%
Therefore, for $\alpha >0$ small enough, the Neumann series%
\[
Q:=\sum_{n=1}^{\infty }\left( -1\right) ^{n}\left( \rho \ast G_{\mu }+\rho
\right) ^{\ast n}
\]%
converges in $\mathbf{B}$ and satisfies
$
\left\Vert Q\right\Vert \leq C\alpha$.
Define $
S:=Q+G_{\mu }+Q\ast G_{\mu }$.

\begin{lemma}
Under the condition of Lemma \ref{Le_2}, if $\alpha >0$ is small enough,
then $S=G_{\lambda }^{{\scriptsize (\Gamma )}}$ for $\lambda <\lambda _{%
\mathrm{cr}}\left( \alpha \right) .$
\end{lemma}

\begin{proof}
One checks that%
\[
S-G_{\lambda }^{{\scriptsize (\Pi )}}-G_{\lambda }^{{\scriptsize (\Pi )}%
}\ast S=0.
\]%
$S$ and $G_{\lambda }^{{\scriptsize (\Pi )}}$ are in $L_{2}\left( \mathbb{R}%
^{d}\right) .$ Taking Fourier transforms $\hat{S},\ \hat{G}_{\lambda }^{%
{\scriptsize (\Pi )}}$, we obtain%
\[
\left( 1+\hat{S}\right) \left( 1-\hat{G}_{\lambda }^{{\scriptsize (\Pi )}%
}\right) =1.
\]%
We have that $G_{\lambda }^{{\scriptsize (\Gamma )}}\in L_{2}\left( \mathbb{R%
}^{d}\right) $ for $\lambda <\lambda _{\mathrm{cr}}\left( \alpha \right) ,$
and from (\ref{Basic_Green_identity})%
\[
\left( 1+\hat{G}_{\lambda }^{{\scriptsize (\Gamma )}}\right) \left( 1-\hat{G}%
_{\lambda }^{{\scriptsize (\Pi )}}\right) =1.
\]%
By Fourier inversion in $L_{2}\left( \mathbb{R}^{d}\right) ,$ the claim
follows.
\end{proof}

\begin{proof}[{\bf Proof of Lemma \protect\ref{Le_2}}]
We prove that $S=S_{\alpha ,\lambda },$ as constructed above, satisfies $%
S_{\alpha ,\lambda }\leq 2G$ for $\lambda <\lambda _{\mathrm{cr}}\left(
\alpha \right) ,$ provided $\alpha $ is small enough. From the previous
lemma, one then also has $G_{\lambda }^{{\scriptsize (\Gamma )}}\leq 2G.$

As $\mu <1$ we have $G_{\mu }\leq G.$ Furthermore, as $\left\Vert
Q\right\Vert \leq C\alpha ,$ one has for $\alpha $ small enough $\left\vert
Q\right\vert \leq G/2.$ It therefore suffices to show that for $\alpha $
small enough, one has%
\[
\left\vert \left( Q\ast G_{\mu }\right) \left( x\right) \right\vert \leq
G\left( x\right) /2.
\]%
For that it suffices that for any function $q\in \mathbf{B}$, with $q\geq 0,$
and $\left\Vert q\right\Vert \leq 1,$ one has $q\ast f_{d}\leq C_{d}f_{d},$
for some $C,$ depending only on $d,$ and $f_{d}\left( x\right) :=\left(
1+\left\vert x\right\vert \right) ^{-d+2}.$ \ Indeed,%
\begin{eqnarray*}
\left( q\ast f_{d}\right) \left( x\right)  &=&\int_{\left\vert
y-x\right\vert \leq \left\vert x\right\vert /2}q\left( y\right) f_{d}\left(
x-y\right) \mathrm{d}y+\int_{\left\vert y-x\right\vert >\left\vert
x\right\vert /2}q\left( y\right) f_{d}\left( x-y\right) \mathrm{d}y \\
&\leq &2^{d}\min \left( 1,\left\vert x\right\vert ^{-d}\right)
\int_{\left\vert z\right\vert \leq \left\vert x\right\vert /2}f_{d}\left(
z\right) \mathrm{d}z+\left\Vert q\right\Vert _{1}f_{d}\left( x/2\right)  \\
&\leq &C_{d}f_{d}\left( x\right) .
\end{eqnarray*}
\end{proof}

\section{Technical Estimates}\label{Sect_Technical_est}\label{sec-technical-est}

\noindent In the proofs of Lemma \ref{Le_1} and Lemma \ref{Le_2} we have used some technical estimates, which we state and prove in this section.
Recall that $\varphi$ is the standard Gaussian density and that the interaction function $v\colon[0,\infty)\to[0,\infty)$ is continuous and satisfies $v\leq L$ and $v(r)=0$ for $r>R$ for some parameters $L,R\in(0,\infty)$.

\begin{lemma}[Gaussian Decay of $u_{n,\be}$]
\label{Le_Est_u} 
Recall the function $u_n=u_{n,\be}$ defined in \eqref{def-un}. Then there exist constants $C,A  \in (0,\infty )$, depending
only on $d,L,R$, such that, for any $n$ and $\mathbf{x}= (
x_{1},x_{2},\ldots ,x_{2n-1},x_{2n} ) \in  ( \mathbb{R}^{d} )
^{2n}$, 
\begin{equation*}
{u}_{n} ( \mathbf{x} ) \leq C^{n}\be^{n-1}\exp  \Big[ -\frac{1}{8%
}\sum\nolimits_{i=1}^{n}|x_{2i}-x_{2i-1}|^{2}-A 
\sum\nolimits_{i=1}^{n-1}|x_{2i}-x_{2i+2}|^{2} \Big] .
\end{equation*}
\end{lemma}

\begin{proof}Let us first assume that $n=2m$ is even (the case of odd $n$ is similar, we drop it in the following), and  identify $u_{2m}(\mathbf x)$ in terms of integrals over standard Brownian bridges $\mathrm{BB}^{\ssup i}$ on $ [0,1 ]$ that start and terminate  at zero. Indeed, if we define $\overline{x}_{t}^{\ssup i}:= ( 1-t ) x_{2i-1}+tx_{2i}$, then $(\overline{x}_{t}^{\ssup i }+\mathrm{BB}^{\ssup {i}} ( t ))_{t\in[0,1] }$ are independent Brownian bridges from $x_{2i-1}$ to $x_{2i}$, $i=1,\dots,m$. That is, they have the same distribution as $\mu^{\ssup 1}_{x_{2i-1},x_{2i}}/\varphi(x_{2i}-x_{2i-1})$. Hence,  we can write
\begin{equation}\label{Uenrepresentation}
u_{2m} ( \mathbf{x} )
=
\prod_{i=1}^{2m}\varphi  (
x_{2i}-x_{2i-1} ) \mathbb{E}\bigg[ \prod_{i=1}^{2m-1} (
1-\exp  [ -\be \varpi _{i} ]  ) \bigg], \\
\end{equation}
where
\begin{equation*}
\varpi _{i}:=\int_{0}^{1}\d t\,\ v \Big( \big|\overline{x}_{t}^{\ssup i }+\mathrm{%
BB}^{\ssup {i}} ( t ) -\overline{x}_{t}^{ \ssup {i+1}}-\mathrm{BB}%
^{\ssup {i+1}} ( t ) \big| \Big)\qquad\mbox{and }\overline{x}_{t}^{\ssup i}= ( 1-t ) x_{2i-1}+tx_{2i}.
\end{equation*}

We will later upper estimate the expectation on the right-hand side of \eqref{Uenrepresentation} against products of expectations over Brownian brides, using the Cauchy--Schwarz  inequality; hence it will be critical and almost sufficient to handle the case $n=2$. Hence, first prove the assertion for the case $n=2$.  We will use the well known fact (proved with the help of the reflection principle) that the distribution of $\sup_{t\in[0,1]} \vert 
\mathrm{BB} ( t )  \vert $ has sub-gaussian tails, i.e., there
exist constants $c_{1},c_{2}>0$, depending only on $d,$ such that%
\begin{equation}
\mathbb{P}\Big( \sup_{t\in[0,1]} \vert \mathrm{BB} ( t )
 \vert >u\Big) \leq c_{1}\exp  [ -c_{2}u^{2} ],\qquad u\in(0,\infty) .
\label{Bridge_tails}
\end{equation}

We first consider the set%
\begin{equation*}
W:=\Big \{\mathbf x\in\R^4\colon \max_{i=1,2} \vert x_{2i}-x_{2i-1} \vert >\frac{1}{4}%
 \vert x_{4}-x_{2} \vert  \Big\} .
\end{equation*}%
If $ \vert x_{2}-x_{1} \vert > \vert x_{4}-x_{2} \vert /4,$
we have%
\begin{equation}
\exp  \Big[ -\frac{1}{2} \vert x_{2}-x_{1} \vert ^{2} \Big] \leq
\exp  \Big[ -\frac{1}{4} \vert x_{2}-x_{1} \vert ^{2}-\frac{1}{64}%
 \vert x_{4}-x_{2} \vert ^{2} \Big] ,  \label{Est_u4_1}
\end{equation}%
and similarly when $ \vert x_{4}-x_{3} \vert > \vert
x_{4}-x_{2} \vert /4.$ Therefore, for $\mathbf x\in W,$ we have, also using that $v\leq L$ and that $1-\e^{-\alpha L}\leq \alpha L$,%
\begin{equation*}
{u}_{2} ( \mathbf{x} ) \leq \frac{L\be }{ ( 2\pi  ) ^{d}}%
\exp  \Big[ -\frac{1}{4} \vert x_{2}-x_{1} \vert ^{2}-\frac{1}{4}%
 \vert x_{4}-x_{3} \vert ^{2}-\frac{1}{64} \vert
x_{4}-x_{2} \vert ^{2} \Big] ,
\end{equation*}
which shows the claim for $n=2$ on $W$.

On $W^{\rm c},$ and on the event
\begin{equation*}
\Big\{\sup_{t\in[0,1]} \vert \mathrm{BB}^{\ssup i} ( t )  \vert \leq 
\frac{1}{2} \vert x_{4}-x_{2} \vert -2R\Big\}
\end{equation*}%
one has $\varpi _{1}=0.$ Therefore, on $W^{\rm c},$%
 \begin{equation}\label{Est_u4_2}
 \begin{aligned}
{u}_{2} ( \mathbf{x} ) &\leq \frac{L\be }{ ( 2\pi  )
^{d}}\exp  \Big[ -\frac{1}{2} \vert x_{2}-x_{1} \vert ^{2}-\frac{1}{%
2} \vert x_{4}-x_{3} \vert ^{2} \Big] 2\mathbb{P} \Big( \sup_{t\in[0,1]} \vert \mathrm{BB} ( t )
 \vert >\frac{1}{2} \vert x_{4}-x_{2} \vert -2R \Big) \\
 &\leq C\be \exp  \Big[ -\frac{1}{2} \vert x_{2}-x_{1} \vert ^{2}-\frac{1%
}{2} \vert x_{4}-x_{3} \vert ^{2}-A   \vert
x_{4}-x_{2} \vert ^{2} \Big],  
\end{aligned}
\end{equation}
using (\ref{Bridge_tails}),
with some constants $C,A  >0,$ depending only on $d,L,R.$ This proves the
claim for $n=2$, actually with slightly better coefficients in the exponent.
But the ones in the statement comes up in an induction argument.

Let $n\geq 3.$ To simplify the notation, we take $n$ even, $n=2m.$ (The
case of an odd $n$ follows by a straightforward modification.) We have 
$$
\begin{aligned}
\frac{u_{2m} ( \mathbf{x} )}
{\prod_{i=1}^{2m}\varphi  (
x_{2i}-x_{2i-1} )} &= \mathbb{E}\bigg[ \prod\nolimits_{j=1}^{2m-1} (
1-\exp  [ -\be \varpi _{j} ]  ) \bigg] \\
&=\mathbb{E}\bigg[
\prod\nolimits_{j=1}^{m} ( 1-\exp  [ -\be \varpi _{2j-1} ]
 ) \prod\nolimits_{j=1}^{m-1} ( 1-\exp  [ -\be \varpi _{2j}%
 ]  ) \bigg] \\
&\leq  \bigg[ \mathbb{%
E}\bigg(\prod\nolimits_{j=1}^{m}\big( 1-\exp  [ -\be \varpi _{2j-1} ]\big)^2\bigg)
 \mathbb{E}\bigg(\prod\nolimits_{j=1}^{m-1}\big( 1-\exp  [ -\be
\varpi _{2j}  ] \big)^2 \bigg) \bigg] ^{1/2} \\
&\leq \bigg[ \mathbb{%
E}\bigg(\prod\nolimits_{j=1}^{m}\big( 1-\exp  [ -\be \varpi _{2j-1} ]\big)\bigg)
 \mathbb{E}\bigg(\prod\nolimits_{j=1}^{m-1}\big( 1-\exp  [ -\be
\varpi _{2j}  ] \big)\bigg) \bigg] ^{1/2} \\
&= \bigg[
\prod\nolimits_{j=1}^{m}\mathbb{E}\Big( 1-\exp  [ -\be \varpi _{2j-1}%
 ] \Big) \prod\nolimits_{j=1}^{m-1}\mathbb{E}\Big( 1-\exp  [
-\be \varpi _{2j} ] \Big) \bigg] ^{1/2}
\end{aligned}
$$
In the third line, we used the Cauchy-Schwarz bound, while in the following inequality, we used that $1-\mathrm{e}^{-\be
\varpi }\in  ( 0,1 ) ,$ and so we can drop the square. In the fifth line we used independence.
Hence, we get
$$
\begin{aligned}
u_{2m} ( \mathbf{x} )
&\leq \Big(\prod_{i=1}^{2m}\sqrt{\varphi  (
x_{2i}-x_{2i-1} )}\Big)
\sqrt{\varphi  ( x_{2}-x_{1} ) }\sqrt{\varphi  (
x_{2m}-x_{2m-1} ) } \\
&\qquad\times \sqrt{\prod\nolimits_{j=1}^{m}\varphi  ( x_{4j}-x_{4j-1} ) \varphi  (
x_{4j-2}-x_{4j-3} ) \mathbb{E} ( 1-\exp  [ -\be \varpi _{2j-1}%
 ]  ) } \\
&\qquad\times \sqrt{\prod\nolimits_{j=1}^{m-1}\varphi  ( x_{4j+2}-x_{4j+1} ) \varphi  (
x_{4j}-x_{4j-1} ) \mathbb{E} ( 1-\exp  [ -\be \varpi _{2j}%
 ]  ) }\\
 &\leq \Big(\prod_{i=1}^{2m}\sqrt{\varphi  (
x_{2i}-x_{2i-1} )}\Big)
\sqrt{\varphi  ( x_{2}-x_{1} ) }\sqrt{\varphi  (
x_{2m}-x_{2m-1} ) } \\
&\qquad\times \sqrt{\prod\nolimits_{j=1}^{m} u_2(x_{4j },x_{4j-1},x_{4j-2},x_{4j-3})} \times \sqrt{\prod\nolimits_{j=1}^{m-1}
u_2(x_{4j+2},x_{4j+1},x_{4j},x_{4j-1})}
\end{aligned}
$$
Now we apply the estimate in \eqref {Est_u4_2} for each of terms in the last line and summarize to finish the proof for $n=2m$. 
\end{proof}

\begin{lemma}
\label{Le_Gauss_convol}
\begin{enumerate}
\item[a)] For any $m\in \mathbb{N}$ and $h>0,$  there exists $C (
m,h ) >0$ such that%
\begin{equation*}
\frac{1}{C ( m,h )  ( 1+ \vert y \vert  ) ^{m}}%
\leq \int_{\mathbb{R}^{d}}\frac{1}{ ( 1+ \vert x \vert  )
^{m}}\mathrm{e}^{-h \vert y-x \vert ^{2}}\,\d x\leq \frac{C (
m,h ) }{ ( 1+ \vert y \vert  ) ^{m}},\qquad y\in\R^d.
\end{equation*}

\item[b)] For any $m_{1},m_{2}\in \mathbb{N}$ and $h>0$, there exists $%
C ( m_{1},m_{2},h ) >0$ such that%
$$
\int_{\mathbb{R}^{d}}\frac{1}{ ( 1+ \vert y_{1}-x \vert
 ) ^{m_{1}}}\frac{1}{ ( 1+ \vert y_{2}-x \vert  )
^{m_{2}}}\mathrm{e}^{-h \vert x \vert ^{2}}\,\d x \leq \frac{C ( m_{1},m_{2},h ) }{ ( 1+ \vert
y_{1} \vert  ) ^{m_{1}} ( 1+ \vert y_{2} \vert
 ) ^{m_{2}}},\qquad y_1,y_2\in\R^d.
$$
\end{enumerate}
\end{lemma}

\begin{proof}
a)
$$
\begin{aligned}
\int_{\mathbb{R}^{d}}\frac{1}{ ( 1+ \vert x \vert  ) ^{m}}%
\mathrm{e}^{-h \vert y-x \vert ^{2}}\,\d x
 &=\int\limits_{ \vert
x-y \vert \leq  \vert y \vert /2}\frac{1}{ ( 1+ \vert
x \vert  ) ^{m}}\mathrm{e}^{-h \vert y-x \vert
^{2}}\,\d x+\int\limits_{ \vert x-y \vert > \vert y \vert /2}%
\frac{1}{ ( 1+ \vert x \vert  ) ^{m}}\mathrm{e}%
^{-h \vert y-x \vert ^{2}}\,\d x \\
&\leq \frac{1}{ ( 1+ \vert y \vert /2 ) ^{m}}\int \mathrm{%
e}^{-h \vert z \vert ^{2}}\,\d z+\int_{ \vert z \vert
> \vert y \vert /2}\mathrm{e}^{-h \vert z \vert ^{2}}\,\d z \\
&\leq 2^{m}\pi ^{d/2}h^{-d/2}\frac{1}{ ( 1+ \vert y \vert
 ) ^{m}}+\int_{ \vert z \vert > \vert y \vert /2}%
\mathrm{e}^{-h \vert z \vert ^{2}}\,\d z,
\end{aligned}
$$%
where we used in the first integration area that $\vert
x-y \vert \leq  \vert y \vert /2$ implies that $|x|\geq |y|-|x-y|\geq |y|/2$. As the second summand on the right-hand side has a Gaussian tail bound, the second inequality follows.

For the first,%
\begin{eqnarray*}
\int_{\mathbb{R}^{d}}\frac{1}{ ( 1+ \vert x \vert  ) ^{m}}%
\mathrm{e}^{-h \vert y-x \vert ^{2}}\,\d x &\geq
&\int\limits_{ \vert x \vert \leq  \vert y \vert +1}\frac{1%
}{ ( 1+ \vert x \vert  ) ^{m}}\mathrm{e}^{-h \vert
y-x \vert ^{2}}\,\d x \\
&\geq &\frac{1}{ ( 2+ \vert y \vert  ) ^{m}}%
\int\limits_{ \vert z \vert \leq 1}\mathrm{e}^{-h \vert
z \vert ^{2}}\,\d z\geq c ( m,h ) \frac{1}{ ( 1+ \vert
y \vert  ) ^{m}}
\end{eqnarray*}%
for some small $c ( m,h ) >0$ and all $y\in\R^d$.

b) follows by applying the  Cauchy--Schwarz inequality and a).
\end{proof}

\begin{lemma}
\label{Le_Edgeworth}$G$ satisfies%
\begin{equation*}
G ( x ) =a_{d} \vert x \vert ^{2-d}+O (  \vert
x \vert ^{-d-2} )\qquad \mbox{as }|x|\to\infty,\qquad\mbox{where }a_{d}:=\frac{\Gamma  ( d/2-1 ) }{2\pi ^{d/2}},
\end{equation*}
and $\Gamma $ is the standard Gamma function.
\end{lemma}

\begin{proof} We approximate $G$ by the Green function in continuous time:%
\begin{equation*}
G_{\mathrm{c}} ( x ) :=\int_{0}^{\infty }\varphi _{t} (
x ) \,\d t=a_{d} \vert x \vert ^{-d+2}.
\end{equation*}%
To determine the correction, we consider 
\begin{equation*}
G_{\mathrm{c}} ( x ) -G ( x ) =\sum_{n=1}^{\infty
}\int_{n-1/2}^{n+1/2} [ \varphi _{t} ( x ) -\varphi _{n} (
x )  ]\,\d t+\int_{0}^{1/2}\varphi _{t} ( x )\,\d t.
\end{equation*}%
The second summand has Gaussian tails  in $ \vert x \vert ,$ so we have
only to care about the first one.  A Taylor expansion gives
\begin{equation}\label{Taylor}
\varphi _{t} ( x ) -\varphi _{n} ( x ) = ( t-n )
\partial _{1}\varphi _{n} ( x ) +\frac{ ( t-n ) ^{2}}{2}%
\partial _{1}^{2}\varphi _{n+\theta  ( t-n ) } ( x ) ,\qquad
\theta \in  [ 0,1 ].
\end{equation}%
Note that
\begin{eqnarray*}
\partial _{1}\varphi _{t} ( x ) &=& \Big[ \frac{ \vert
x \vert ^{2}}{2t^{2}}-\frac{d}{2t} \prod_{i=1}^{2m}\varphi  (
x_{2i}-x_{2i-1} )\Big] \varphi _{t} ( x ) \\
\partial _{1}^{2}\varphi _{t} ( x ) &=& \prod_{i=1}^{2m}\varphi  (
x_{2i}-x_{2i-1} )\Big[ -2\frac{ \vert
x \vert ^{2}}{2t^{3}}+\frac{d}{2t^{2}} \Big] \varphi _{t} (
x ) + [ \frac{ \vert x \vert ^{2}}{2t^{2}}-\frac{d}{2t}%
 \Big] ^{2}\varphi _{t} ( x ) \\
&=&\frac{1}{t^{2}} \Big[  \Big( \frac{ \vert x \vert ^{2}}{t}%
 \Big) ^{2}-\frac{ \vert x \vert ^{2}}{t} ( 1+\smfrac{d}{2}%
 ) +\frac{d}{2} ( 1+\smfrac{d}{2} )  \Big] \varphi _{t} (
x )
\end{eqnarray*}
The integration $\int_{n-1/2}^{n+1/2}\,\d t$ over the first summand on the right-hand side of \eqref{Taylor} is $0.$
Therefore%
$$
\begin{aligned}
\sum_{n=1}^{\infty }\Big \vert \int_{n-1/2}^{n+1/2} [ \varphi _{t} ( x ) -\varphi
_{n} ( x )  ] \,\d t \Big\vert
& \leq \frac{1}{24}\sum_{n=1}^{\infty }\sup_{t\in  [
n-1/2,n+1/2 ] } \vert \partial _{1}^{2}\varphi _{t} ( x )
 \vert\\
 &\leq
\sum_{n=1}^{\infty }\frac{C}{n^{2+d/2}} \Big(  \Big( \frac{ \vert
x \vert ^{2}}{n} \Big) ^{2}+1 \Big) \exp  \Big[ -\frac{ \vert
x \vert ^{2}}{2 ( n+1 ) ^{2}} \Big] \\
&=O (  \vert x \vert ^{-d-2} )
\end{aligned}
$$%
as $ \vert x \vert  \to \infty .$
\end{proof}

\noindent{\bf Acknowledgement.} The research of the third author is funded by the Deutsche Forschungsgemeinschaft (DFG) under Germany's Excellence Strategy EXC 2044-390685587, Mathematics M\"unster: Dynamics-Geometry-Structure. The authors are indebted to Gordon Slade for very valuable feedback on the earlier version of the paper.

\end{document}